\documentclass[11pt,parskip=half]{scrartcl}
\usepackage[a4paper,hmargin={2.5cm,2.5cm},vmargin={2.5cm,2.5cm},heightrounded, marginparwidth=2.2cm, marginparsep=0.1cm]{geometry}

\usepackage[utf8]{inputenc}
\usepackage[T1]{fontenc}

\usepackage[leqno]{amsmath}
\usepackage{amssymb,amsthm}
\usepackage{tikz-cd}
\usepackage{xypic}

\theoremstyle{plain}
\newtheorem{theorem}{Theorem}[section]
\newtheorem{lemma}[theorem]{Lemma}
\newtheorem{proposition}[theorem]{Proposition}
\newtheorem{corollary}[theorem]{Corollary}

\theoremstyle{definition}
\newtheorem{definition}[theorem]{Definition}
\newtheorem{examples}[theorem]{Examples}
\newtheorem{example}[theorem]{Example}

\theoremstyle{remark}
\newtheorem{remark}[theorem]{Remark}
\newtheorem*{remark*}{Remark}
\newtheorem{remarks}[theorem]{Remarks}

\RequirePackage[shortlabels]{enumitem}

\newlist{tfae}{enumerate}{1}%
\setlist[tfae,1]{label=(\roman*)}%

\newcommand{\catfont}[1]{\mathsf{#1}}

\newcommand{\SET}{\catfont{Set}}

\newcommand{\AB}{\catfont{Ab}}
\newcommand{\TOP}{\catfont{Top}}

\newcommand{\Arr}{\mathsf{Arr}}

\DeclareMathAlphabet{\mathmybb}{U}{bbold}{m}{n}

\newcommand{\ncatfont}[1]{\mathbb{#1}}

\newcommand{\ncatC}{\ncatfont{C}}

\newcommand{\CE}{\mathcal{E}}
\newcommand{\CF}{\mathcal{F}}

\newcommand{\CK}{\mathcal{K}}

\newcommand{\CM}{\mathcal{M}}
\newcommand{\CN}{\mathcal{N}}

\newcommand{\CP}{\mathcal{P}}

\newcommand{\CW}{\mathcal{W}}

\DeclareMathOperator{\Span}{Span}
\DeclareMathOperator{\Cosp}{Cosp}
\DeclareMathOperator{\Pb}{Pb}
\DeclareMathOperator{\Po}{Po}

\DeclareMathAlphabet{\mathpzc}{OT1}{pzc}{m}{it}

\newcommand{\op}{\mathrm{op}}

\RequirePackage{dsfont}

\newcommand{\CC}{\field{C}}
\newcommand{\field}[1]{\mathds{#1}}

\RequirePackage{mathtools}

\makeatletter

\def\slashedarrowfill@#1#2#3#4#5{%
  $\m@th\thickmuskip0mu\medmuskip\thickmuskip\thinmuskip\thickmuskip
   \relax#5#1\mkern-7mu%
   \cleaders\hbox{$#5\mkern-2mu#2\mkern-2mu$}\hfill
   \mathclap{#3}\mathclap{#2}%
   \cleaders\hbox{$#5\mkern-2mu#2\mkern-2mu$}\hfill
   \mkern-7mu#4$%
}

\newcommand*{\rightmodarrowfill@}{\slashedarrowfill@\relbar\relbar{\raisebox{0pc}{$\hspace{1pt}\circ$}}\rightarrow}
\newcommand*{\xmodto}[2][]{\ext@arrow 0055{\rightmodarrowfill@}{\;#1\;}{\;#2\;}}

\makeatother

\title{The normal decomposition of a morphism\\ in categories without zeros}

\author{Renier Jansen, Muhammad Qasim, Walter Tholen\footnote{Work on this paper began during visits by R. Jansen and M. Qasim at York University in 2023/24. R. Jansen acknowledges support from the University of the Free State for his visit, and W. Tholen acknowledges partial financial assistance under the NSERC Discovery Grants Program (no. 501260) in support of M. Qasim's visit.}  }
\publishers{}
\date{\today}

\usepackage[hypertexnames=false]{hyperref}

\hypersetup{
  colorlinks = true,
  citecolor= [rgb]{0,0.75,0}, 
  urlcolor=[rgb]{0.4,0,0.4}, 
  linkcolor=[rgb]{0.8,0,0} 
}

\begin{document}

\maketitle

\begin{abstract}
\footnotesize{ For a morphism $f:A\to B$ in a category $\CC$ with sufficiently many finite limits and colimits, we discuss an elementary construction of a functorial decomposition
$$\xymatrix{A\ar@{->>}[r] & P_f\ar[r] & N_f\;\ar@{>->}[r] & B}$$
which, if $\CC$ happens to have a zero object, amounts to the standard decomposition
$$\xymatrix{A\ar@{->>}[r] & \mathrm{Coker(ker}f)\ar[r] & \mathrm{Ker(coker}f)\;\ar@{>->}[r] & B}.$$
In this way we obtain natural notions of normal monomorphism and normal epimorphism also in non-pointed categories, as special types of regular mono- and epimorphisms. We examine the factorization behaviour of these classes of morphisms in general, compare the generalized normal decompositions with other types of threefold factorizations, and illustrate them in some every-day categories. The concrete construction of normal decompositions in the slices or coslices of some of these categories can be challenging. Amongst many others, in this regard we consider particularly  the categories of T$_1$-spaces and of groups.} 
\end{abstract}

{\footnotesize \tableofcontents}

{\footnotesize {\em Mathematics Subject Classification:} 18A20; !8B30, 18A32, 18F60.

}

\section{Introduction}
\label{sec:some-title}
The purpose of this paper is to demonstrate that the decomposition
$$\xymatrix{A\ar@{->>}[r] & \mathrm{Coker(ker}f)\ar[r] & \mathrm{Ker(coker}f)\;\ar@{>->}[r] & B}$$
of a morphism $f:A\to B$ in a pointed category with kernels and cokernels allows for a useful generalization to the non-pointed context, as provided by any category $\CC$ with (certain) finite limits and colimits. The elementary construction we give leads to natural notions of {\em normal monomorphism} and {\em normal epimorphism} in $\CC$, as a specialization of the notions  of regular mono- and epimorphism, and they assume their usual meaning if $\CC$ is pointed. Hence, also in the non-pointed case, every morphism $f$ factors canonically through a normal monomorphism as a second factor, which we call its {\em normal closure}, as well as through a normal epimorphism as a first factor, called the {\em normal dual closure} of $f$. This terminology caters to the role models for $\CC$, pointed or not. In the pointed category $\mathsf{Grp}$ of groups, the normal closure of $f:A\to B$ is of course given by the least normal subgroup of $B$ containing the image of $f$ (in this paper denoted by $\widehat{\mathrm{Im}(f)}^B$), and in the non-pointed category $\TOP_1$ of T$_1$-spaces it is given by the usual topological closure of the image of $f$ in $B$ (denoted by $\overline{\mathrm{Im}(f)}^B$).

Assigning to a morphism $f$ in $\CC$ its normal closure can be seen as reflecting the object $f$ of the arrow category $\mathsf{Arr}(\CC)$ into the full subcategory $\mathsf{NMono}(\CC)$ with objects all normal monomorphisms in $\CC$.
We therefore begin this article by recalling some known facts about arbitrary classes $\CM$ of morphisms in $\CC$ which, considered as full subcategories of $\mathsf{Arr}(\CC)$, are replete and reflective. Such a class $\CM$ is, of course, closed under limits and therefore stable under pullback, but it is closed under composition only if $\CM$ is part of an orthogonal factorization system $(\CE,\CM)$ in $\CC$.

Our first theorem (Theorem \ref{firsttheorem}) establishes the reflectivity of the class $\CN=\mathsf{NMono}(\CC)$ of normal monomorphisms in $\CC$, and dually the coreflectivity of the class $\CP=\mathsf{NEpi}(\CC)$ of all normal epimorphisms. Combining the two emerging factorizations leads to the {\em normal decomposition} 
$$\xymatrix{A\ar@{->>}[r]^{\pi_f} & P_f\ar[r]^{\kappa_f} & N_f\;\ar@{>->}[r]^{\nu_f} & B}$$
of a morphism $f$, amounting to the cokernel-kernel decomposition in the pointed case. We call the normal decompositions in $\CC$ {\em perfect} if both classes, $\CN$ and $\CP$, are closed under composition. This, of course, is not the case in the pointed (in fact semi-abelian \cite{JMT02}) category of groups, but it is in the pointed category of commutative monoids and in the pointed (quasi-, or almost-, abelian \cite{Rump01}) category
of topological abelian groups, as well as in the non-pointed categories of commutative unital rings and of  T$_1$-spaces.

If normal decompositions in $\CC$ are perfect, the classes $\CN$ and $\CP$ lead to orthogonal factorization systems $(\mathcal{C},\CN)$ and $(\CP,\CF)$, which allow us to pinpoint a ``natural home'' for the {\em comparison morphisms} $\kappa_f$ occurring in the normal decompositions: they are precisely the morphisms in the class $\CK=\mathcal{C}\cap\CF$ (Theorem \ref{perfecttheorem}). The triple $(\CP,\CK,\CN)$ then becomes an {\em orthogonal threefold factorization system} of $\CC$, called a ``double factorization system'' in \cite{PultrTholen02}. We recall the properties of such systems in full generality in Section 5, and show how they bijectively relate to pairs of orthogonal factorization systems $(\mathcal{C},\CF\cap\CW), (\mathcal{C}\cap\CW,\CF)$, for a suitable class $\CW$. The thus obtained triples $(\mathcal{C},\CW,\CF)$ give a {\em Quillen model structure} \cite{Quillen67} for $\CC$ precisely when $\CW$ satisfies the 2-for-3 property which, however, places a severe constraint on the threefold system $(\CP,\CK,\CN)$ (Proposition \ref{2-for-3}).

Beginning from Section 6 we investigate normal decompositions in the (co)slice categories of a `parent category', since (co)slice categories tend to be `highly non-pointed', even when the parent category is the prototypical pointed and, in fact abelian, category $\mathsf{Ab}$ of abelian groups. Indeed, since our general parent category $\CC$ has an initial object, for any object $C$ the slice $\CC/C$ can be pointed only if $C$ is a zero object in $\CC$, in which case $\CC/C\cong\CC$. It is easy to see that for $\CC=\mathsf{Ab}$ and any object $C$, our generalized notions of normal mono- and epimorphism give the expected normal decompositions of morphisms also in the generally non-pointed categories $\mathsf{Ab}/C$ and $C/\mathsf{Ab}$, so that they are carried by their decompositions in $\mathsf{Ab}$ (see Examples \ref{normalclosuresliceexas}(2) and \ref{cosliceSetRing}(2)). 

However, for general $\CC$ we show that the normal closure $\nu_{f/C}$ of a morphism
$$\xymatrix{A\ar[rr]^f\ar[rd]_q && B\ar[ld]^p\\
& C &\\}$$
in $\CC/C$ must be carefully distinguished from the normal closure $\nu_f$ of $f$ taken in $\CC$. 
Generally speaking, the computation of the normal closure in the slices of many concrete categories may become increasingly challenging, simply because more general pushout diagrams are involved than in the ordinary case, {\em i.e.}, when $C$ is terminal in $\CC$. We give descriptions of the normal closure for morphisms of the slices of the categories of T$_1$-spaces  (Theorem \ref{slicedT1closure})
and of groups (Theorem \ref{slicedGrpclosure}) which show strong formal similarities despite the obvious discrepancies between the two parent categories. Concretely, for $f$ in $\mathsf{Grp}/C$ the normal closure is given by the subgroup inclusion
$$\bigcup_{a\in A}f(a)\widehat{E}^B\hookrightarrow B\qquad\text{with }E=\mathrm{Ker}(p)\cap \mathrm{Im}(f),$$
and for  $f$ in $\TOP_1/C$ by the subspace inclusion
$$\quad\bigcup_{a\in A}\overline{E_a}^B\hookrightarrow B\qquad\text{with } E_a=p^{-1}(pa)\cap\mathrm{Im}(f).$$

Finally, in Section 9 we briefly mention that taking normal closures and normal dual closures in a category $\CC$ and its (co)slices fits into the more general context provided by the adjunction between the spans in $\CC$ over a fixed pair of objects and the cospans under this pair, given by taking pushouts and pullbacks of spans and cospans, respectively. The unit of this adjunction specializes to the reflection of a morphism to its normal closure when one of the objects in the pair is terminal in $\CC$, and likewise in the dual situation. The case of a discrete normal closure is then described by so-called {\em Doolittle diagrams}, {\em i.e.} squares that are simultaneously pushout and pullback diagrams.

We leave it for future work to extrapolate from the many example categories discussed in this paper to obtain further insights on the normal closure and its dual in more general types of categories, pointed or not. In the pointed case, there are recent results in \cite{GGHVV24} giving characteristic conditions for a semi-abelian category making its class of normal monomorphisms closed under composition, and in \cite{PeschkeVdL24}, Remark 3.2.2, the relation of this property with the Snake Lemma has been emphasized. We are grateful to Tim Van der Linden for these literature pointers after a presentation of this paper by the third author at a conference in honour of Enrico Vitale in Milan in September 2024. As suggested by Vitale at that occasion, it would certainly be desirable to investigate the theme of this paper also in a 2-categorical context. Indeed, we have left a discussion of the prototypical 2-category, $\mathsf{Cat}$, for future work, even when considered just as an ordinary category.

\section{Reflective and coreflective classes of morphisms}

Let $\CM$ be a class of morphisms in a category $\CC$. We call $\CM$ {\em replete} if $\CM$ contains all identity morphisms of $\CC$ and is closed under composition with isomorphisms from either side; equivalently, if $\CM$ contains the class of isomorphisms of $\CC$ and, considered as a full subcategory of the category $\Arr(\CC)$ of morphisms of $\CC$, is closed under isomorphism, {\em i.e.}, is replete in the standard sense for full subcategories. We use $\langle u,v\rangle:f\to g$ as our notation for a morphism in $\Arr(\CC)$, given by the commutative diagram
$$\xymatrix{A\ar[rr]^u\ar[d]_f && C\ar[d]^g\\
B\ar[rr]_v && D\\}$$
in $\CC$.

\begin{definition}
A replete class $\CM$ of $\CC$ is {\em reflective} if $\CM$ is reflective when considered as a full subcategory of $\Arr(\CC)$ with $\CM$ as its class of objects. Dually, a replete class $\CE$ of $\CC$ is {\em coreflective} if it is a reflective class of $\CC^{\op}$.
\end{definition}

\begin{remark}\label{firstremark}
A reflective class $\CM$ is closed under the formation of limits in $\Arr(\CC)$. As such, it is {\em stable under pullback} in $\CC$ and {\em closed under wide pullbacks} (intersections) in $\CC$. Furthermore, $\CM$ satisfies the {\em weak left cancellation property}
$$n\cdot m\in\CM\quad\text{and}\quad (n\in\CM\text{ or } n\; \text{monic})\quad \Longrightarrow\quad m\in\CM\; .$$
The dual properties hold for coreflective classes. For details see \cite{ImKelly86} and \cite{DT95} (Theorem 1.7, Exercise 1.G).
\end{remark}

Here is how one may detect reflectivity of a class:

\begin{proposition}\label{reflchar}
A  replete class $\CM$ of $\CC$ is reflective if, and only if, every morphism $f$ in $\CC$ factors as $f=n_f\cdot \hat{f}$ with $n_f\in\CM$ such that, whenever $m\cdot u=v\cdot f$ with $m\in\CM$, one has $t\cdot \hat{f}=u$ and $m\cdot t=v\cdot n_f$, for a uniquely determined morphism $t$.
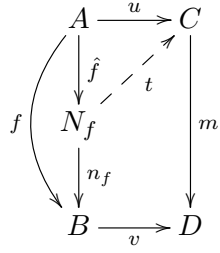
$$\xymatrix{
		A \ar[r]^-{u} \ar[d]^-{\hat{f}}\ar@/_1.5pc/[dd]_f & C \ar[dd]^m \\
		N_{f} \ar@{-->}[ur]_-{t} \ar[d]^-{n_{f}} & \\
		B \ar[r]_-{v} & D	
	}$$
\end{proposition}

\begin{proof}
The given condition is sufficient for the reflectivity of $\CM$. Indeed, $\langle \hat{f}, 1_B\rangle: f\to n_f$ serves as the reflection of $f$ into $\CM$, since the above diagram, redrawn as 
$$
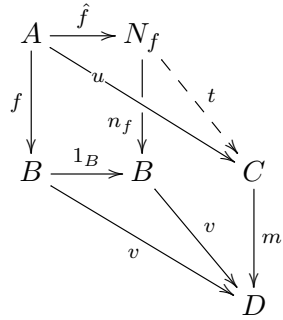
\xymatrix@R=1.3cm{
		A \ar[d]_-{f} \ar[r]^-{\hat{f}} \ar[drr]|(.3){~u~} & N_{f} \ar|\hole[d]_(.65){n_{f}} \ar@{-->}[dr]^-{t} & \\
		B \ar[r]^{1_{B}} \ar[drr]_-{v} & B \ar[dr]^-{v} & C \ar[d]^-{m} \\
		& & D
	}$$
actually shows its universal property. Conversely, having a reflection $\langle \hat{f}, i\rangle: f\to n $ of $f$ into $\CM$, it suffices to show that necessarily $i$ is an isomorphism, since then  $n_f:=i^{-1}\cdot n\in \CM$ (by repleteness) satisfies the condition of the Proposition. To this end, writing $i:B\to E$, we first observe that the morphism $\langle f,1_B\rangle: f\to 1_B$ corresponds to $\langle k,j\rangle:n\to 1_B$ by adjunction, where $k=j\cdot n$ and $j\cdot i=1_B$. Then $i\cdot j=1_E$ follows from the fact that the morphisms $\langle i\cdot k, i\cdot j\rangle, \langle n, 1_E\rangle: n\to 1_E$ are both equalized by the reflection $\langle \hat{f}, i\rangle$.
\end{proof}

\begin{corollary}
For a reflective class $\CM$ in $\CC$, in the notation of the Proposition, one has
$$\CM=\{f\in\mathsf{mor}(\CC)\mid \hat{f}  \text{ {\em is an isomorphism}} \}.$$
\end{corollary}

\begin{remarks}\label{secondremark}
(1) Proposition \ref{reflchar} appeared in its dual form as  Proposition 1.3 in \cite{MacDTholen82} and as 
Corollary 3.3 in \cite{Tholen83}, with the needed factorizations referred to as {\em locally orthogonal}. In the terminology of \cite{DT95}, Section 1.5 and Exercise 1.G, it shows that a class $\CM$ is reflective in $\CC$ if, and only if, $\CC$ has {\em right $\CM$-factorizations}.

(2) Recall (\cite{FreydKelly72, Pumplun72}) that for any classes $\CM$ and $\CE$ of morphisms in $\CC$ one may form the {\em left orthogonal complement} $^\bot\CM=\{f\in\mathrm{mor}(\CC)\mid \forall m\in \CM:f\bot m\}$ of $\CM$ and the {\em right orthogonal complement} $\CE^{\bot}=\{g\in\mathrm{\CC}\mid\forall e\in\CE: e\bot g\}$ of $\CE$; here $f\bot g$ stands for the unique diagonalization (or lifting) property: whenever $g\cdot u=v\cdot f$, then $t\cdot f=u$ and $g\cdot t=v$ for a unique morphism $t$. In light of the Corollary, it is noteworthy that the left orthogonal complement of a reflective class $\CM$ is analogously described as
$$^\bot\CM=\{f\in\mathrm{mor}(\CC)\mid n_f\;\text{is an isomorphism}\}. $$
\end{remarks}

Recall that, by definition, for replete\footnote{It suffices to have $\mathsf{Iso}(\CC)\cdot\CE\subseteq\CE$  and $\CM\cdot \mathsf{Iso}(\CC)\subseteq \CM$.} classes $\CE$ and $\CM$, the pair $(\CE,\CM)$ is an {\em orthogonal factorization system (o.f.s.)} of $\CC$ if $\CM\cdot\CE= \mathsf{mor}(\CC)$ and $\CE\bot\CM$. The classes $\CE$ and $\CM$ determine each other since $\CE=\,^{\bot}\CM$ and $\CM=\;\CE^{\bot}$.
Well-known necessary consequences are that $\CM$ must be closed under composition and under the formation of limits in $\Arr(\CC)$. In fact, $\CM$ must be a reflective class of $\CC$ and, dually, $\CE$ must be coreflective, so that these classes respectively enjoy the properties mentioned in Remark \ref{firstremark} and their duals. It is well known that the converse statement needs closure under composition (see Proposition 2.1 of \cite{Tholen83}); for the reader's convenience, we indicate the proof, as follows.

\begin{proposition}\label{ofschar}
A replete class $\CM$ of a category $\CC$ belongs to an o.f.s. $(\CE,\CM)$ if, and only if, $\CM$  is reflective and closed under composition.
\end{proposition}

\begin{proof}
Since the ``only if'' part is well known, we confirm here only the ``if'' part. Hence, having a reflection $\langle \hat{f}, 1_B\rangle: f\to n_f$ of $f:A\to B$ into $\CM$, it suffices to show that the morphism $g:=\hat{f}$ belongs to the class $\CE:= {}^{\bot}\!\CM$. But this will follow from Proposition \ref{reflchar} once we have shown that the morphism $\hat{g}$ of the reflection of $g$ into $\CM$ is an isomorphism. To this end specialize the diagram of Proposition \ref{reflchar} to 
$$\xymatrix{
		A \ar[r]^-{\hat{g}} \ar[d]_-{\hat{f}}\ar@/_2.0pc/[dd]_f & N_g \ar[dd]^{n_f\cdot n_g} \\
		N_{f} \ar@{-->}[ur]_-{t} \ar[d]_-{n_{f}} & \\
		B \ar[r]_-{1_B} & B	
	}$$
where $n_f\cdot n_g\in\CM$ by hypothesis on $\CM$. The morphism $t$ with $t\cdot \hat{f}=\hat{g}$ and $n_f\cdot n_g\cdot t=n_f$ actually satisfies $n_g\cdot t=1_{N_f}$, by the uniqueness property of the factorization $f=n_f\cdot \hat{f}$. Likewise, the uniqueness property of the factorization $g=n_g\cdot\hat{g}$ shows $t\cdot n_g=1_{N_g}$, which concludes the proof.
\end{proof}

In its dual form, Proposition \ref{ofschar} implies the following ``folklore'' fact, stated under minimal hypotheses on the category:
\begin{corollary}\label{regular}
In a category $\CC$ with kernel pairs and their coequalizers, the class of regular epimorphisms is coreflective. Moreover, every morphism in $\CC$  factors (regular epi, mono) if, and only if, regular epimorphisms are closed under composition in $\CC$.
\end{corollary}

\section{Normal monomorphisms and epimorphisms, the normal closure}
For simplicity, we now {\em assume that the category $\CC$ be finitely complete and finitely cocomplete, denoting the terminal object by $1$ and the initial object by $0$}. For a morphism $f:A\to B$, consider the following diagrams, where $N_f=B\times_{(B+_A1)}1$ with pullback projection $\nu_f$, and $P_f= 0+_{(0\times_BA)}A$ with pushout injection $\pi_f$. Note that, as a pullback of a split monomorphism, $\nu_f$ is a regular monomorphism; dually, $\pi_f$ is a regular epimorphism.

$$
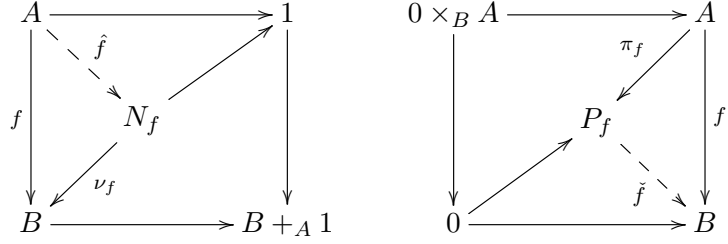
\xymatrix{A\ar[dd]_{f}\ar[rr]\ar@{-->}[rd]^{\hat{f}} && 1\ar[dd]\\
	& N_f \ar[ld]^{\nu_f}\ar[ru] & \\
	B \ar[rr] && B+_A  1\\} \qquad
	\xymatrix{0\times_{B} A \ar[dd]_{}\ar[rr]^{} && A \ar[dl]_-{\pi_{f}} \ar[dd]^-{f}\\
	& P_f \ar@{-->}[dr]_-{\check{f}} & \\
	0 \ar[rr] \ar[ur] && B\\} $$
	
	Equivalently, having formed the pushout $B+_A1$ one may construct the morphism $\nu_f$ as the equalizer of the pushout injection $B\to B+_A1$ and the composite morphism $B\to 1\to B+_A1$. Likewise, one may present the morphism $\pi_f$ as a coequalizer.

\begin{definition}
We call the morphism
	$f$ a {\em normal monomorphism} if the induced morphism $\hat{f}$ with $\nu_f\cdot \hat{f}=f$ is an isomorphism; equivalently, if the left outer pushout square is also a pullback diagram. Dually, $f$ is a {\em normal epimorphism} if the induced morphism $\check{f}$ with $\check{f}\cdot\pi_f=f$ is an isomorphism or, equivalently, if the right outer pullback square is also pushout diagram. This respectively defines the classes $\mathsf{NMono}(\CC)$ and $\mathsf{NEpi}(\CC)$ of normal monomorphisms and epimorphisms of $\CC$.
	\end{definition}
	
	As noted above, normality implies regularity, {\em i.e.}
	$$\mathsf{NMono}(\CC)\subseteq\mathsf{RMono}(\CC)\quad\text{and}\quad \mathsf{NEpi}(\CC)\subseteq \mathsf{REpi}(\CC)\,. 	$$
Furthermore, the assignment $f\mapsto N_f$ clearly defines a functor $\Arr(\CC)\to \CC$ which sends $\langle u,v\rangle:f\to g$ to the unique morphism $t$ rendering the following diagram on the left commutative. The diagram on the right, with $s= v+_u1$, gives a more elaborate picture of its construction as $t=v\times_s1$.
$$	
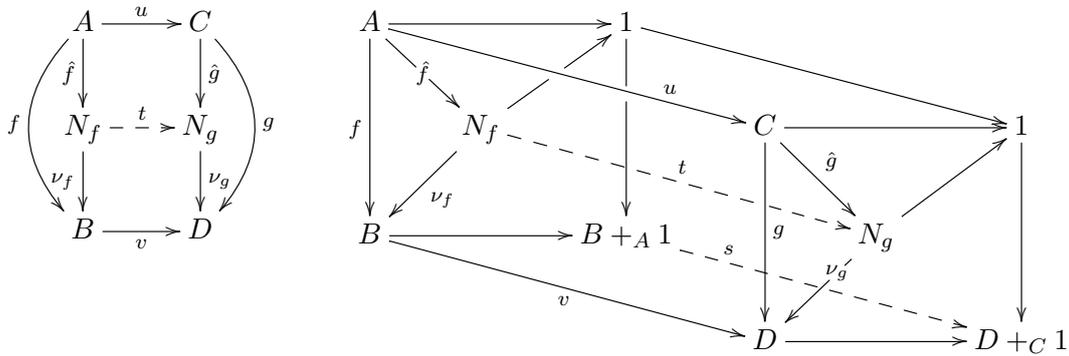
\xymatrix{
		A \ar[d]_-{\hat{f}} \ar[r]^-{u} \ar@/_1.7pc/[dd]_-{f} & C \ar[d]^-{\hat{g}} \ar@/^1.7pc/[dd]^-{g} \\
		N_{f} \ar[d]_-{\nu_{f}} \ar@{-->}[r]^-{t} & N_{g} \ar[d]^-{\nu_{g}} \\
		B \ar[r]_-{v} & D }
\qquad \xymatrix{
	A \ar[rr]^-{} \ar[dr]|-{\hat{f}} \ar[dd]_{f} \ar[drrr]^(0.75){u} &  & 1 \ar|(.33)\hole[dd]^-{} \ar[drrr]^-{} &  &  &  \\
	& N_{f} \ar|(.52)\hole[ur]^-{} \ar[dl]^-{\nu_{f}} \ar@{-->}|(.71)\hole[drrr]^-{t} &  & C \ar[rr]^(.4){} \ar[dr]^-{\hat{g}} \ar[dd]^-{g} &  & 1 \ar[dd]^-{} \\
	B \ar[rr]^-{} \ar[drrr]_{v} &  & B+_{A}1 \ar@{-->}[drrr]^(.25){s} &  & N_{g} \ar[ur]^-{} \ar[dl]|(.35){\nu_{g}} &  \\
	&  &  & D \ar[rr]^-{} &  & D+_{C}1		
}$$
We are now ready to prove:

\begin{theorem}\label{firsttheorem}
The class $\mathsf{NMono}(\CC)$ is reflective in $\CC$, and the class $\mathsf{NEpi}(\CC)$ is coreflective.
\end{theorem}

\begin{proof}
For every morphism $f$, we must prove that $g:=\nu_f$ is a normal monomorphism, {\em i.e.} that $\hat{g}$ is an isomorphism. For that, in the diagram above, consider $u:=\hat{f}$ and $v:=1_B$. By the pullback property of $N_f$, there is a morphism $r:N_g\to N_f$ with $\nu_f\cdot r=\nu_g$.  Envoking also the pullback property of $N_g$ one sees that $r$ is inverse to $\hat{g}$.

In order to show that the factorization $f=\nu_f\cdot\hat{f}$ enjoys the universal property as described in Proposition \ref{reflchar}, consider again the commutative diagrams above, but now with any $g\in\mathsf{NMono}(\CC)$ and $u,v$ arbitrary. Since $\hat{g}$ is an isomorphism, clearly $t':=\hat{g}^{-1}\cdot t$ will serve as the fill-in morphism as required by the Proposition.
\end{proof}

With Remark \ref{firstremark} one concludes:

\begin{corollary}
The class $\mathsf{NMono}(\CC)$ is stable under pullback and closed under wide pullbacks in $\CC$, and it has the weak right cancellation property. The class $\mathsf{NEpi}(\CC)$ has the dual properties.
\end{corollary}

Briefly: Normal mono- and epimorphisms enjoy the same properties known for the larger classes of regular mono- and epimorphisms---and they do so for the same reasons: reflectivity and coreflectivity (see Corollary \ref{regular}).

For the computation of $N_f,\,\nu_f$ and $P_f,\,\pi_f$ in concrete categories, the following trivial observation is useful:

\begin{lemma}\label{triviallemma}
If $f=m\cdot e$ with an epimorphism $e$, then there is a unique isomorphism $i:N_f\to N_m$ with $\nu_f=\nu_m\cdot i$, so that  $\nu_f\cong\nu_m$. 

Dually, if $f=m\cdot e$ with a monomorphism $m$, then there is a unique isomorphism $j:P_e\to P_f$ with $\pi_f=j\cdot \pi_e$, so that $\pi_f\cong\pi_e$.
\end{lemma}

Consequently, if every morphism in $\CC$ factors (epi, mono), it suffices to determine $N_m$ and $\nu_m$ for all monomorphisms $m$. If the factorizations are actually (strong epi, mono), so that for every morphism $f:A\to B$ and all subobjects $k$ of $A$ one has an (up to iso unique) image $f(k)$ in $B$ obtained by factoring $f\cdot k$,  the assignment
$$k\longmapsto \nu_k$$
defines in fact an idempotent {\em closure operator} of $\CC$ with respect to the class of all monomorphisms (in the sense of \cite{DT95}); that is, in the usual (pre)order for subobjects, the operator $\nu$ is extensive and monotone,
and (as a consequence of the functoriality of $f\mapsto N_f$ alluded to above) it satisfies the so-called {\em continuity condition}
$$f(\nu_k)\leq \nu_{f(k)}\,.$$
The monomorphisms closed with respect to this closure operator are precisely the normal monomorphisms.

Dually, in the presence of (epi, mono)-factorizations it suffices to compute $\pi_e$ for all epimorphisms $e$, and if the factorizations are actually (epi, strong mono), the assignment
$e\longmapsto \pi_e$ defines an idempotent {\em dual closure operator} with respect to the class of all epimorphisms (in the sense of \cite{DT15}). The epimorphisms closed with respect to this dual closure operator are precisely the normal epimorphisms.

Consequently, also in the general context of our finitely complete and finitely cocomplete category $\CC$, without any {\em a priori} factorization hypotheses, for any morphism $f$ we will refer to $\nu_f$ as the {\em normal closure} of $f$ and to $\pi_f$ as the {\em normal dual closure} of $f$. In the following examples we illustrate their construction. The characterization of normal mono- and epimorphisms follows in the next section.

\begin{examples}\label{firstexas} (1) If the category $\CC$ is {\em pointed}, {\em i.e.} if $0\cong 1$, then
$$ \nu_f\cong \mathrm{ker(coker}f)\quad\text{and}\quad\pi_f\cong\mathrm{coker(ker}f)$$
for all morphisms $f:A\to B$. Obviously then, for $\CC=\mathsf{Ab}$ the category of Abelian groups, $\nu_f:N_f\to B$ may naturally be given by  the inclusion map $\mathrm {Im}f=f(A)\hookrightarrow B$, and $\pi_f:A\to P_f$ by the quotient map $A\to A/\mathrm{Ker}f$; of course, here $N_f\cong P_f$.
By contrast, for $\CC=\mathsf{Grp}$ the category of all groups, whereas $P_f$ may be computed as in $\mathsf{Ab}$, now $N_f$ is the least normal subgroup of $B$ containing $\mathrm {Im}f\leq B$.

(2) In the category $\mathsf{CMon}$ of commutative (multiplicatively written) monoids and their unit-preserving homomorphisms, the pushout $B+_A1$ may be computed as the quotient monoid $B/\!\sim$ with respect to the congruence relation
$$x\sim y\iff \exists\; a,b\in A: (fa)x=(fb)y$$
on $B$. Hence, with $e$ neutral  in $B$ and the projection $p:B\to B/\!\!\sim$, one obtains the normal closure $\nu_f:N_f\hookrightarrow B$ with
$$N_f=p^{-1}(pe)=\{x\in B\mid \exists\; a,b\in A: (fa)x=fb\}=\{x\in B\mid f(A)x\cap f(A)\neq\emptyset\}\;.$$
(We note that, if every element of $f(A)$ is invertible in $B$, in particular if $B$ happens to be a group, then $N_f=\{\frac{fa}{fb}\in B\mid a,b\in A\}$ is the symmetrization of $f(A)$ in $B$.) The normal dual closure $\pi_f: A\to P_f=A/\!\!\simeq$ is obtained similarly, with the congruence relation $\simeq$ on $A$ given by
$$ a\simeq b\iff \exists \;u,v\in\mathrm{Ker}f=f^{-1}e: ua=vb\,.$$

(3) In the (non-pointed) category $\mathsf{CRing}_1$ of commutative unital rings and their unital homomorphisms, the terminal object is strict (every out-going morphism is an isomorphism), so that the normal closure of a morphism $f:A\to B$ may be taken to be the identity morphism on $B$. For the normal dual closure of $f$, with the initial object given by the integers $\mathbb Z$, one has $C:=\mathbb Z\times_B A=\{(n,x)\mid x\in A, n\in \mathbb Z, fx=n1_B\}$. Since $\{x-n1_A\mid (n,x)\in C\}=\mathrm{Ker}f$, one may construct the pushout $P_f$ as $A/\mathrm{Ker}f\cong\mathrm{Im}f$.

(4) In the (non-pointed) category $\mathsf{Set}$ of sets, the normal closure of $f$ may be given (as in  $\mathsf{Ab}$) by $f(A)\hookrightarrow B$ while $\pi_f$ may be taken to be the identity map on $A$. By contrast, in the pointed category $\mathsf{Set}_*$ of pointed sets $(A,a)$ with $a\in A$ and point-preserving maps, with zero object $(1=\{*\},*)$, while the normal closure may be given as in $\mathsf{Set}$, the normal dual closure of $f:(A,a)\to(B,b)$ may be taken to be the projection $\pi_f:(A,a)\to(A/\!\!\sim, \pi_f(a))$ where $\sim$ identifies precisely any pair of distinct points in $f^{-1}b$. Alternatively, one may present the quotient set as $(1+(A\setminus f^{-1}b),*)$.

(5) In the category $\mathsf{Top}$ of topological spaces, for the normal closure of $f$ one lets the set $N_f=f(A)$ carry the subspace topology of $B$. Its normal dual closure trivializes like in $\mathsf{Set}$.
 In the category $\mathsf{Top}_*$ of pointed topological spaces, one may proceed as in $\mathsf{Set}_*$, providing the set $P_f=A/\!\!\sim$ with the quotient topology of $A$.

(6) In the category $\mathsf{Top}_1$ of T$_1$-spaces, the normal dual closure of $f:A\to B$ is again trivial, {\em i.e.} $\pi_f\cong\mathrm{id}_A$. However, unlike in $\TOP$, the normal closure $\nu_f$ may now be given by the usual {\em Kuratowski closure} $\overline{f(A)}$ of $f(A)$ in $B$, {\em i.e.} $\nu_f: \overline{f(A)}\hookrightarrow B$. This is clear for $A=\emptyset$, in which case the pushout $B+_A1$ in $\TOP_1$ is just the topological sum $B+1$. If $A\neq\emptyset$, we fix $a_0\in A$ and see that the pushout $B+_A1$ in $\TOP_1$ may be constructed as the quotient space $B/\!\sim$, where now $\sim$ identifies precisely all pairs of points in $\overline{f(A)}$. 

Indeed, this quotient space is certainly 
T$_1$ since any non-singleton fibre $p^{-1}(px)$ of the projection $p:B\to B/\!\!\sim$ must equal the closed set  $\overline{f(A)}$ in $B$, and since the singleton fibres are trivially closed in the T$_1$-space $B$. 
Furthermore, to show that 
$$\xymatrix{A\ar[r]\ar[d]_f& \mathsf 1\ar[d]^{pa_0}\\
B\ar[r]^{p} & B/\!\sim\\
}$$ 
has the universal property of a pushout, it suffices to see that any continuous map $g:B\to C$ in $\TOP_1$ that is constant on $f(A)$ must actually be constant on $\overline{f(A)}$. But this is an immediate consequence of the continuity of $g$ and the T$_1$-property of $C$. Indeed, $$g(\overline{f(A)})\subseteq \overline{g(f(A))}=g(f(A)).$$
(We will repeat this argumentation in a more general context in the proof of Proposition \ref{T1pushout}.)

The transition from $\TOP_1$ to the pointed category $(\TOP_1)_*$ proceeds like the transition from $\TOP$ to $\TOP_*$.

(7) To obtain the normal closure and normal dual closure of a morphism in the category $\mathsf{TopAb}$ of topological Abelian groups ({\em i.e.} of group objects in $\TOP$) one provides the constructions done in $\mathsf{Ab}$ with the subspace and quotient topologies, respectively. When $\mathsf{Ab}$ is replaced by $\mathsf{Grp}$, {\em i.e.} in the category $\mathsf{TopGrp}$, the dual normal closure is formed like in $\mathsf{TopAb}$, and for the normal closure of a morphism one provides the least normal subgroup containing the image with the subspace topology. 

(8) Considering only Hausdorff group topologies, we have the same normal dual closure constructions in the categories $\mathsf{HausAb}$ and $\mathsf{HausGrp}$ as in the non-Hausdorff case, but to form the normal closure of $f:A\to B$, in $\mathsf{HausAb}$ we must provide the topological closure of $\mathrm{Im}f$ in $B$ with the subspace topology, and in $\mathsf{HausGrp}$ we must first form the least normal subgroup of $\mathrm{Im}f$ in $B$ and then take the topological closure in $B$.
\end{examples}

\section{The normal decomposition of a morphism}
To form the desired decomposition of a morphism $f:A\to B$ in a finitely complete and finitely cocomplete category $\CC$, in the notation of Section 3 we first consider the factorization $f=\nu_f\cdot\hat{f}$ through its normal closure. Since $\nu_f$ is monic, by Lemma \ref{triviallemma}, the morphism $\hat{f}$ has  the same normal dual closure as $f$, up to a unique isomorphism $i:P_{\hat{f}}\to P_f$ satisfying $i\cdot\pi_{\hat{f}}=\pi_f$ and $\check{f}\cdot i=\nu_f\cdot\check{\hat{f}}$. Dually, the morphism $\check{f}$ belonging to the normal dual closure factorization of $f$ has the same normal closure as $f$, up to a unique isomorphism $j:N_f\to N_{\check{f}}$ satisfying $j\cdot \hat{f}=\hat{\check{f}}\cdot\pi_f$ and $\nu_{\check{f}}\cdot j=\nu_f$. Since $\pi_{\hat{f}}$ is epic (or $\nu_{\check{f}}$ is monic), one also has $j\cdot \check{\hat{f}}=\hat{\check{f}}\cdot i$.
$$\xymatrix{& P_{\hat{f}}\ar[rr]^{\check{\hat{f}}}\ar[dd]_(0.65){i}|(.33)\hole|(.5)\hole && N_f \ar[rd]^{\nu_f}\ar[dd]^(0.35){j}|(.5)\hole|(.68)\hole & \\
A\ar[ru]^{\pi_{\hat{f}}}\ar[rd]_{\pi_f}\ar[rrrr]|f\ar[rrru]|{\hat{f}}&&&& B\\
& P_f\ar[rr]_{\hat{\check{f}}}\ar[rrru]|{\check{f}} && N_{\check{f}}\ar[ru]_{\nu_{\check{f}}} &\\
}$$
Setting $\kappa_f:=\check{\hat{f}}\cdot i^{-1}=j^{-1}\cdot\check{\hat{f}}$ we obtain the following commutative diagram on the left. Of course, $\kappa_f$ is also determined as the unique diagonal morphism rendering the diagram on the right commutative:
$$\xymatrix{&P_f\ar[rr]^{\kappa_f} && N_f\ar[rd]^{\nu_f} & \\
A\ar[rrrr]^f\ar[ru]^{\pi_f} &&&& B \\}
\qquad\qquad
\xymatrix{A\ar[rr]^{\hat{f}}\ar[d]_{\pi_f} && N_f\ar[d]^{\nu_f}\\
P_f \ar[rr]_{\check{f}}\ar[rru]|{\kappa_f} && B\\}
$$
\begin{definition}
We refer to $f=\nu_f\cdot\kappa_f\cdot\pi_f$ as the {\em normal decomposition} of $f$ and call $\kappa_f$ its {\em comparison morphism}.
\end{definition}

By Theorem \ref{firsttheorem}, for all morphisms $f$ one has
$$\pi_f\in \CP:=\mathsf{NEpi}(\CC)\quad\text{and}\quad\nu_f\in\CN:=\mathsf{NMono}(\CC).$$
This raises the question whether the comparison morphisms $\kappa_f$ have a ``natural home'' as well. Only in good cases do we have a convenient description of such home, as follows.
\begin{theorem}\label{perfecttheorem}
{\em (1)} The following are equivalent:
\begin{itemize}
\item[\em (i)] The class $\CN$ is closed under composition.
\item[\em (ii)] The class $\CN$ belongs to an o.f.s. $(\mathcal C,\CN)$.
\item[\em (iii)] Every comparison morphism lies in $^{\bot}\CN$.
\end{itemize}
{\em (2)} The following are equivalent:
\begin{itemize}
\item[\em (i)] The class $\CP$ is closed under composition.
\item[\em (ii)] The class $\CP$ belongs to an o.f.s. $(\CP,\CF)$.
\item[\em (iii)] Every comparison morphism lies in $\CP^{\bot}$.
\end{itemize}
{\em (3)} Every comparison morphism lies in $^{\bot}\CN\cap\CP^{\bot}$ if, and only if, the classes $\CN$ and $\CP$ are closed under composition; equivalently, if for all morphisms $f$,
$$\kappa_f\in\CK:=\{k\in\mathrm{mor}(\CC)\mid \nu_k\; \text{and}\; \pi_k\;\text{are isomorphisms}\}\,,$$
in which case we call the normal decompositions in $\CC$ {\em perfect}.
\end{theorem}

\begin{proof}
(1) The equivalence (i)$\iff$(ii) follows from Proposition \ref{ofschar} and Theorem \ref{firsttheorem}. For (i)$\Rightarrow$(iii), we note that the proof of Proposition \ref{ofschar} shows $\hat{f}=\kappa_f\cdot\pi_f\in\, ^{\bot}\CN$ for all morphisms $f$. Since $\pi_f$ is epic, this implies $\kappa_f\in\,^{\bot}\CN$, by the dual of the cancellation rule of Remark \ref{firstremark}.
For (iii)$\Rightarrow$(ii), since trivially $\pi_f\in \,^{\bot}\CN$, and since every (left or right) orthogonal complement is closed under composition, one has $\hat{f}=\kappa_f\cdot\pi_f\in\,^{\bot}\CN$. Hence, the normal closure of $f$ provides the $(^{\bot}\CN,\CN)$-factorization of $f$.

(2) is the dualization of (1). The first part of (3) follows from (1) and (2), and for the equivalent formulation given in its second part see also Remark \ref{secondremark}(2), which implies $\CK=\, ^{\bot}\CN\cap\CP^{\bot}$.
\end{proof}

\begin{corollary}
 If all comparison morphisms in $\CC$ are epic, then every strong monomorphism in $\CC$ is normal. Dually:
if all comparison morphisms in $\CC$ are monic, then every strong epimorphism in $\CC$ is normal.
\end{corollary}

\begin{proof}
If $\kappa_f$ is epic, then so is $\hat{f}=\kappa_f\cdot \pi_f$, for all morphisms $f$. So, by hypothesis, $\CC$ has (necessarily orthogonal) (epi, normal mono)-factorizations, which makes the class $\CN$ the right orthogonal complement of the class of epimorphisms. But by definition, this class coincides with the class of strong monomorphisms.
\end{proof}

A category $\CC$ is {\em quasi-abelian} if $\CC$ is additive and has kernels and cokernels, and if normal epimorphisms are stable under pullback and normal monomorphisms are stable under pushout; see \cite{Schneiders99}. Quasi-abelian categories are called {\em almost abelian} in \cite{Rump01} but the concept may be traced back to \cite{Yoneda60}. It has been shown in \cite{Rump01} that in such categories all comparison morphisms are both epic and monic. Therefore:

\begin{corollary}\label{quasi-abelian}
If the category $\CC$ is quasi-abelian, every strong mono- or epimorphism in $\CC$ is normal, and its normal decompositions are perfect.
\end{corollary}

In the general context of our finitely complete and finitely cocomplete category $\CC$ we finally emphasize the functoriality of the passage from a morphism to its normal decomposition. More precisely, by the coreflectivity of $\CP=\mathsf{NEpi}(\CC)$ and the reflectivity of $\CN=\mathsf{NMono}(\CC)$ in $\mathsf{Arr}(\CC)$ we have endofunctors $\pi$ and $\nu$ of $\mathsf{Arr}(\CC)$ given by
$$ f\longmapsto \pi_f\quad\text{and}\quad f\longmapsto \nu_f,$$
which actually carry the structure of a comonad and a monad, respectively. Composing the coreflection of $f$ with the reflection of $f$ in $\mathsf{Arr}(\CC)$ defines a natural transformation $\alpha:\pi\Longrightarrow \nu$, {\em i.e.} componentwise one has 
 $$\alpha_f=\langle\hat{f},\check{f}\rangle:\pi_f\longrightarrow\nu_f\,.$$
 
 \begin{proposition}
 Assigning to every morphism $f$ in $\CC$ its comparison morphism $\kappa_f$ defines an endofunctor of $\mathsf{Arr}(\CC)$ through which the natural transformation $\alpha$ factors, i.e, there are natural transformations $\rho:\pi\Longrightarrow\kappa$ and $\sigma:\kappa\Longrightarrow\nu$ with $\alpha=\sigma\cdot\rho$.
 \end{proposition}
 
 \begin{proof}
 Given a morphism $\langle u,v\rangle:f\to g$ in $\mathsf{Arr}(\CC)$, consider the following diagram on the left:
 $$\xymatrix{A\ar[rr]^u\ar[d]_{\pi_f}\ar@/_22pt/[dd]_{\hat{f}} && C\ar[d]^{\pi_g}\\
 P_f\ar[rr]^{P_{u,v}}\ar[d]_{\kappa_f} && P_g\ar[d]^{\kappa_g}\ar@/^22pt/[dd]^{\check{g}}\\
 N_f\ar[rr]^{N_{u,v}}\ar[d]_{\nu_f} && N_g\ar[d]^{\nu_g}\\
 B\ar[rr]^v && D\\}
 \qquad\qquad\xymatrix{ && \\ 
 A\ar[r]^{\pi_f}\ar[d]_{\pi_f}\ar@/^20pt/[rr]^{\hat{f}} & P_f\ar[d]^{\kappa_f }\ar[r]^{\kappa_f}& N_f\ar[d]^{\nu_f}\\
 P_f\ar[r]^{\kappa_f}\ar@/_20pt/[rr]_{\check{f}} & N_f\ar[r]^{\nu_f} & B\\
 }$$
 Here the morphism $N_{u,v}$ satisfies its defining identities $N_{u,v}\cdot \hat{f}=\hat{g}\cdot u$ and $\nu_g\cdot N_{u,v}=v\cdot \nu_f$; likewise, $P_{u,v}$ satisfies $P_{u,v}\cdot \pi_f=\pi_g\cdot u$ and $\check{g}\cdot P_{u,v}=v\cdot \check{f}$. Since $\nu_g$ is monic (and $\pi_f$ is epic), the middle square on the left commutes, and we can define the functor $\kappa:\mathsf{Arr}(\CC)\to\mathsf{Arr}(\CC)$ on morphisms by $\kappa\langle u,v\rangle=\langle P_{u,v}, N_{u,v}\rangle$. With the trivially commuting diagram on the right, one defines $\rho_f:=\langle \pi_f,\kappa_f\rangle:\pi_f\to \kappa_f$ and $\sigma_f:=\langle \kappa_f,\nu_f\rangle:\kappa_f\to\nu_f$, so that $\sigma_f\cdot\rho_f=\alpha_f$. The naturality check for $\rho$ and $\sigma$ is a routine matter as well.
 \end{proof}
 
We now return to Examples \ref{firstexas}.

\begin{examples}\label{firstexamplesagain}
(1) Trivially, all mono- and epimorphisms in $\mathsf{Ab}$ are normal, and normal decompositions are perfect, with bijective comparison morphisms: see Corollary \ref{quasi-abelian}. In  the semi-abelian (\cite{JMT02, BorceuxBourn04}) category $\mathsf{Grp}$, however, while all epimorphisms are normal, normal monomorphisms must isomorphically be given by embeddings of normal subgroups, which generally fail to be closed under composition. Therefore normal decompositions in $\mathsf{Grp}$ fail to be perfect.

(2) In $\mathsf{CMon}$, normal epimorphisms are given by surjections $f:A\to B$ with the weak injectivity property $(fa=fb\Longrightarrow a\simeq b)$ for all $a,b\in A$; they are closed under composition. (They are also stable under pullback, which makes the category $\mathsf{CMon}$ {\em 0-normal}: see \cite{Messora24}.) Normal monomorphisms are isomorphically given by submonoid embeddings $A\hookrightarrow B$ satisfying the cancellation property $(ax\in A, a\in A \Longrightarrow x\in A)$ for all $x\in B$. They are closed under composition as well, so that normal decompositions in $\mathsf{CMon}$ are perfect.

(3) In $\mathsf{CRing}_1$, all surjections are normal epimorphisms, but only isomorphisms are normal monomorphisms. Normal decompositions are perfect, with injective comparison morphisms.

(4) All injections in $\SET$ are normal monomorphisms, but only bijections are normal epimorphisms. With surjective comparison morphisms, normal decompositions are perfect. This is true also in $\SET_*$, where normal monomorphisms are described as in $\SET$. The normal epimorphisms $f:(A,a)\to (B,b)$ are characterized as those surjections whose restriction to $A\setminus f^{-1}b$ is injective, and the comparison morphisms are those surjections $f$ with $|f^{-1}b|=1$.
 
(5) In $\TOP$ and $\TOP_*$ one just topologizes the normal decompositions of $\SET$ and $\SET_*$,  using subspace and quotient space topologies. The normal decompositions stay perfect.

(6) In $\TOP_1$ the normal monomorphisms are isomorphically characterized as embeddings of closed subspaces---not of all subspaces as in $\TOP$. But like in $\TOP$, normal epimorphisms in $\TOP_1$ are confined to be homeomorphisms, making the normal decompositions again perfect, with all comparison maps having a dense image in their codomains.
This remains true in the pointed category $(\TOP_1)_*$, where normal monomorphisms are characterized as in $\TOP_1$, but where normal epimorphisms are all quotient maps which map $A\setminus f^{-1}b$ injectively.

(7) In the quasi-abelian category $\mathsf{TopAb}$, normal mono- and epimorphisms are isomorphically characterized as topological subgroup embeddings and as topological quotient homomorphisms, respectively. The normal decompositions are perfect, with continuous bijective homomorphisms as comparison morphisms. The characterization of normal epimorphisms doesn't change if we move to $\mathsf{TopGrp}$, but here only the topological embeddings of (algebraically) normal subgroups characterize normal monomorphisms, and normal decompositions are no longer perfect.

(8) Also in $\mathsf{HausAb}$ and $\mathsf{HausGrp}$ normal epimorphisms are described as topological quotient homomorphisms. The normal monomorphisms in $\mathsf{HausAb}$ are characterized as closed subgroup embeddings, and normal decompositions are perfect, where injective continuous homomorphisms with dense image serve as comparison morphisms. In $\mathsf{HausGrp}$ one must consider embeddings of closed normal subgroups to characterize normal monomorphisms, and normal decompositions are no longer perfect.
\end{examples}

\section{Orthogonal threefold factorization systems}

In this section we briefly discuss how the normal decompositions in a category $\CC$ with finite limits and finite colimits fare in the context of orthogonal threefold factorization systems as introduced in \cite{PultrTholen02} under the name ``double factorization systems''.\footnote{We avoid this terminology as it may be mistaken to concern factorization systems for a double category.}

\begin{definition}
One calls a triple $(\CP,\CK,\CN)$ of replete\footnote{It suffices to assume $\mathrm{Iso}(\CC)\cdot \CP\subseteq \CP,\; \mathrm{Iso}(\CC)\cdot\CK\cdot\mathrm{Iso}(\CC)\subseteq\CK$ and $\CN\cdot\mathrm{Iso}(\CC)\subseteq\CN$.} classes of morphisms an {\em orthogonal threefold factorization system (o.t.f.s.)} of $\CC$ if $\CN\cdot\CK\cdot\CP=\mathrm{mor}(\CC)$ and $(\CP,\CK)\bot(\CK,\CN)$, where the latter symbolism means that for every outer commutative diagram\
$$\xymatrix{\cdot\ar[rr]^u\ar[d]_p && \cdot\ar[d]^{k'}\\
\cdot\ar[d]_k\ar@{-->}[rru]_s && \cdot\ar[d]^n\\
\cdot\ar[rr]_v\ar@{-->}[rru]^t && \cdot\\
}$$
with $p\in\CP, k,k'\in\CK, n\in\CN$ we have uniquely determined morphisms $s,t$ rendering the entire diagram commutative.
\end{definition}

\begin{remarks}\label{odfsremarks}
(1) For every o.f.s. $(\CE,\CM)$ in $\CC$ one has the trivial o.t.f.s. $(\CE,\mathrm{Iso}(\CC),\CM)$. 

(2) More generally, for every bifibration $P:\CC\to\mathbb B$ and every o.f.s. $(\CE,\CM)$ of $\mathbb B$ one has (in self-explanatory notation) the o.t.f.s. $(\mathrm{Cocart}(P)\cap P^{-1}\CE, P^{-1} \mathrm{Iso}(\mathbb B), \mathrm{Cart}(P)\cap P^{-1}\CM)$. For example, $\TOP$ inherits from the (epi,mono)-factorizations in $\SET$ the o.t.f.s. (quotient maps, continuous bijections, embeddings).

(3) For every o.t.f.s. $(\CP,\CK,\CN)$ of $\CC$ one has the o.f.s. $(\CP,\CF)$ with $\CF:=\CN\cdot\CK$ and the o.f.s. $(\mathcal C,\CN)$ with $\mathcal C:=\CK\cdot\CP$, which are comparable by inclusion: $\CP\subseteq\mathcal C$ and (equivalently) $\CN\subseteq\CF$. As part of an  o.f.s., the class $\CN$ is closed under composition and has the closure properties listed in Remark \ref{firstremark}, and $\CP$ enjoys the dual properties.

(4) For every pair $(\CP,\CF),\;(\mathcal C,\CN)$ of o.f.s. in $\CC$ with $\CP\subseteq\mathcal C$ and $\CN\subseteq\CF$ one obtains the o.t.f.s. $(\CP,\CK,\CN)$ with $\CK:=\mathcal C\cap\CF=\;^{\bot}\CN\cap\,\CP^{\bot}$. We note that the class $\CK$ is closed under composition.

(5) By Theorem 2.7 of \cite{PultrTholen02}, with the assignments described as in (3) and (4), one obtains for any category $\CC$ a bijective correspondence between pairs of comparable orthogonal factorization systems and orthogonal threefold factorization systems.
\end{remarks}

From Theorem \ref{perfecttheorem} we deduce:

\begin{corollary}\label{comparisonmorphisms}
In the finitely complete and finitely cocomplete category $\CC$ one has the o.t.f.s $(\mathsf{NEpi}(\CC),\CK,\mathsf{NMono}(\CC))$ with $\CK=\,^{\bot}\mathsf{NMono}(\CC)\,\cap\,\mathsf{NEpi}(\CC)^{\bot}$ if, and only if, normal decompositions in $\CC$ are perfect.
\end{corollary}
Expanding on Remarks \ref{odfsremarks}(3), (4) we state the following Proposition, which rephrases Remarks 3.11(2) of \cite{PultrTholen02}:

\begin{proposition}
{\em(1)} Given an o.t.f.s $(\CP,\CK,\CN)$ in a category $\CC$, setting
$$\mathcal C:=\CK\cdot\CP,\quad\CW:=\CN\cdot\CP,\quad\CF:=\CN\cdot\CK$$
one obtains orthogonal factorization systems
$$ (\mathcal C, \CF\cap\CW),\quad (\mathcal C\cap\CW,\CF)\quad\text{with}\quad\CW=(\CF\cap\CW)\cdot(\mathcal C\cap\CW)\;.$$
{\em (2)} Conversely, given any triple $(\mathcal C,\CW,\CF)$ of replete classes in $\CC$ satisfying the properties above, one obtains the orthogonal threefold factorization system
$$ (\mathcal C\cap\CW,\,\mathcal C\cap\CF,\,\CF\cap\CW)\;.$$
{\em(3)} The assignments described in {\em (1)} and {\em (2)} are inverse to each other.
\end{proposition}

\begin{proof}
(1) As noted in Remarks \ref{odfsremarks}(3), we have the o.f.s. $(\CP,\CF)$ and should now show $\CP=\mathcal C\cap\CW$. Indeed, ``$\subseteq$'' is trivial since all of our classes contain the isomorphisms. For the reverse inclusion, consider $f=n\cdot p\in\mathcal C$ with $n\in \CN, p\in \CP$. Since $\CP\subseteq\mathcal C$, the weak right cancellation property for $\mathcal C$ (being part of the o.f.s. $(\mathcal C, \CN))$ gives $n\in\mathcal C\cap\CN$, so that $n$ is iso and, hence, $f\in\CP$. This gives the o.f.s. $(\mathcal C\cap\CW,\CF)$ and, dually, the o.f.s. $(\mathcal C,\mathcal F\cap\CW)$. The stated formula for $\CW$ is now an immediate consequence of its definition.
For the proofs of (2) and (3) one proceeds similarly.
\end{proof}
In a nutshell, we may treat a triple $(\mathcal C, \CW,\CF)$ as describing an o.t.f.s. in a kind of inflated fashion. The choice of notation suggests to think of ``cofibrations'', ``weak equivalences'', and of ``fibrations'', and this raises the question: when does $(\mathcal C, \CW,\CF)$ provide us with a {\em Quillen model structure} \cite{Quillen67} on $\CC$? That is: when are $ (\mathcal C, \CF\cap\CW)$ and $(\mathcal C\cap\CW,\CF)$ {\em weak (orthogonal) factorization systems} (see \cite{AHRT02a, AHRT02b}), with $\CW$ satisfying the so-called {\em 2-out-of-3 property} (or better the {\em 3-for-2 property}, so that whenever two of the morphisms in $f=g\cdot h$ are in $\CW$, also the third is)? A {\em weak factorization system} may be defined like an o.f.s., except that the uniqueness part of the diagonalization/lifting property gets dropped everywhere. An o.f.s is automatically also a weak factorization system (a statement that needs proof, see 14.6(3) in \cite{AHS90}). So, the question remains, when does the class $\CW=\CN\cdot\CP$ induced by an o.t.f.s $(\CP,\CK,\CN)$ satisfy the 2-for-3 property? We record here the answer given in \cite{PultrTholen02}, Theorem 3.10, omitting its straightforward proof, as follows:

\begin{proposition}\label{2-for-3}
The triple $(\mathcal C, \CW,\CF)$ corresponding to an o.t.f.s. $(\CP,\CK,\CN)$ in a category $\CC$ gives a Quillen model structure on $\CC$ if, and only if,

{\em 1.} $\CP\cdot\CN\subseteq\CN\cdot\CP$;\; {\em 2.} $(p\cdot k\in \CP,\,p\in\CP, k\in\CK\Rightarrow k \text{ iso })$;\; {\em 3.} $(k\cdot n\in\CN,\, n\in \CN, k\in\CK\Rightarrow k \text{ iso })$.
\end{proposition}

Here Condition 1 guarantees closure under composition of $\CW$ while Conditions 2 an 3 ensure that $\CW$ satisfies the other two components of the 2-for-3 property  . These conditions show that it is rare for perfect normal decompositions to induce a Quillen model structure. In fact, revisiting some of the Examples \ref{firstexamplesagain} we note:

\begin{examples}
The normal decomposition in $\mathsf{Ab}$ induces the (trivial) Quillen model structure (surjective, all, injective). In $\mathsf{CRing}_1$ the induced triple ($\mathcal{C},\CW,\CF)$ satisfies Conditions 1 and 3, but not 2, while in $\SET$ it satisfies 1 and 2, but not 3. In $\TOP$, only Condition 1 holds, and in $\mathsf{TopAb}$ all three conditions fail.
\end{examples}

\section{Normal decomposition in slice and coslice categories}

We fix an object $C$ in our category $\ncatC$ and continue to {\em let $\CC$ be be finitely complete and finitely cocomplete}, in order to  investigate first the normal closure and its dual in the slice category $\ncatC/C$ (whose objects are $\ncatC$-morphisms $p$ with codomain $C$, and whose morphisms $f:q\to p$ are given by $\ncatC$-morphisms $f:A\to B$ with $p\cdot f=q$. Recall that
\begin{itemize}
\item the identity morphism $1_C:C\to C$ of $\CC$ assumes the role of the terminal object in $\ncatC/C$;
\item $\ncatC/C$ is finitely complete, with pullbacks formed as in $\ncatC$;
\item $\ncatC/C$ is finitely cocomplete, with all colimits formed as in $\ncatC$; in particular, the initial object is given by the morphism $0\to C$, with $0$ initial in $\CC$.
\end{itemize}

Since the initial object in $\CC/C$ may be formed with the help of the initial object in $\CC$, it is easy to see that the normal {\em dual} closure of a morphism $f:q\to p$ in $\CC/C$ with $f:A\to B$ in $\CC$ may (in the notation of Section 3) be given by the normal dual closure $\pi_f:A\to P_f$ of $f$ in $\CC$, to be considered as a morphism $\pi_f:p\cdot f\to p\cdot \check{f}$ in $\CC/C$; that is: one may consider the defining diagram for the formation $\pi_f$ in $\CC$ as a diagram over $C$, without changing the role of $0$ and of the pullback and the pushout involved. To briefly indicate this fact, we use the symbolism
$$\pi_{f/C}\cong\pi_f\,.$$

In consideration of the special nature of its terminal object, an analogous formula for the normal closure in $\CC/C$ can {\em not} be expected to hold.

Indeed, in order to form the normal closure of the morphism $f:q\to p$  in $\ncatC/C$ with $f:A\to B$ and $p:B\to C$ in $\ncatC$, we must first form the pushout of $f:p\cdot f\to p$ along the unique arrow $p\cdot f\to 1_C$ in $\ncatC/C$,  given by $p\cdot f$ in $\ncatC$. This then amounts to simply forming the pushout $j_f$  of $f$ along $p\cdot f$ in $\ncatC$, which we may do in two stages, first pushing out $f$ along itself to obtain $f_2$, and then pushing $f_2$ further along $p$ to finally obtain $j_f$:
$$\xymatrix{A\ar[r]^f\ar[d]_f & B\ar[rr]^p\ar[d]_{f_2}\ar@/^2.5pc/|(.5)\hole[dd]^(0.75){1_B} && C\ar[d]_{j_f}\ar@/^2.5pc/[dd]^{1_C}\\
B\ar[r]^{f_1\quad}\ar@/_0.9pc/[rd]_{1_B} & B\!+\!_A\,B\ar[rr]^(.56){1_B+\!_A\,p}\ar[d]_{\nabla} && B\!+\!_A \,C\ar[d]_r\\
&B\ar[rr]_p && C \\ 
}$$ 
Here all rectangles of the above diagram are pushout diagrams in $\mathbb C$, so that $(f_1,f_2)$ is the cokernel pair of $f$. The morphisms $f_1, f_2$ and $j_f$ are all split monomorphisms, even in $\ncatC/B$, with retractions $\nabla=\langle 1_B,1_B\rangle$ and $r=\langle p,1_C\rangle$, respectively. We note that the rectangle on the left is also a pullback diagram if, and only if, $f$ is a regular monomorphism (in $\CC$ or, equivalently, in $\CC/C$).

The following statements are now evident:
\begin{proposition}\label{discreteincomma}
Let $f:q=p\cdot f\to p$ be a morphism in $\ncatC/C$ where $f:A\to B$ in $\CC$.

{\em (1)} The normal closure $\nu_{f/C}$ of $f$ in $\ncatC/C$ is obtained by the pullback of $j_f$ along $\widetilde{p}\cdot f_1$ in $\ncatC$, with $\widetilde{p}:=1_B+_Ap$; we write $N_{f/C}\to B$ for its underlying morphism in $\CC$. If $f$ is a regular monomorphism and the right upper or lower rectangle of the diagram above is a pullback, then $\nu_{f/C}\cong f$.

{\em (2)} There is a unique regular monomorphism $\tau:N_{f/C}\to N_f$ in $\CC$ with $\nu_f\cdot \tau=\nu_{f/C}$ which compares the normal closure of $f$ taken in $\CC/C$ with its normal closure taken in $\CC$. Hence, $\nu_{f/C}\cong\nu_f$ in $ \CC/B$ if, and only if,
$\tau$ is an epimorphism, and this holds whenever
(not only the left but also) the right square of the following diagram is a pullback:
$$\xymatrix{N_{f/C}\ar[rr]\ar[dd]_{\nu_{f/C}} \ar[rd]^{\tau}&& C\ar[rr]\ar[dd]^(0.7){j_f} &&1\ar[dd]\\
& N_f\ar[rrur]|(.37)\hole\ar[ld]_{\nu_f} &&& \\
B\ar[rr]^{\widetilde{p}\cdot f_1} && B+_AC\ar[rr]^{1_B+!_C} && B+_A1\\
}$$
\end{proposition}

We can now compare the normal decompositions in $\CC/C$ and in $\CC$, as follows:

\begin{corollary}\label{normdecincomma}
For a morphism $f$ as in the Proposition one has the commutative diagram 
$$\xymatrix{& P_{f}\ar[rr]^{\kappa_f}\ar[dd]_(0.65){\cong}|(.5)\hole && N_f \ar[rd]^{\nu_f}& \\
A\ar[ru]^{\pi_{f}}\ar[rd]_{\pi_{f/C}}\ar[rrrr]^f\ar@/_2.9pc/[rrdd]_q&&&& B\ar@/^2.9pc/[lldd]^p\\
& P_{f/C}\ar[rr]_{\kappa_{f/C}} && N_{f/C}\ar[ru]_{\nu_{f/C}}\ar[uu]_(0.65){\tau}|(.5)\hole &\\
&& C &&\\
}$$
\end{corollary}
Here is an example showing that $N_{f/C}$ and $N_f$ may differ to the largest extent possible, {\em i.e.}, one may have  $N_{f/C}\cong A$ or $N_f\cong B$, as well as some examples of categories $\CC$ for which the normal closure of a regular subobject in any slice $\CC/C$ is {\em discrete}, so that $\nu_{f/C}\cong f$ for every regular monomorphism $f$ in $\CC/C$:
\begin{examples}\label{normalclosuresliceexas}
(1) Consider the inclusion map $f:\mathbb Q\hookrightarrow \mathbb R$ as a morphism in $\mathsf{Top}_1/\mathbb R$, with the real line $\mathbb R$ taken identically as a space over itself. Subspace embeddings are known to be regular monomorphisms in $\mathsf{Top}_1$ (see \cite{KMPT83}), so that the left upper rectangle in the diagram above Proposition \ref{discreteincomma} is a pullback. With the right side of the diagram being trivial, this gives $N_{f/C}=\mathbb Q$ (as also follows more generally from Corollary \ref{T1corollary} ), while $N_f=\overline{\mathbb Q}=\mathbb R$ by Examples \ref{firstexas}(6).

(2) It is easy to see that, in the category $\mathsf{Ab}/C$ with any abelian group $C$, the normal closure is discrete, like in $\mathsf{Ab}$ itself. Indeed, for $A\leq B$ and any homomorphism $p: B\to C$, the pushout of the inclusion morphism $A\hookrightarrow B$ along $q=p|_A$ is the obvious morphism $C\to (B\oplus C)/I$, where (with $B$ and $C$ embedded into $B\oplus C$)\, $I=\{a-p(a)\mid a\in A\}$, and its pullback along the other pushout map gives back $A\hookrightarrow B$. 

Consequently, for an arbitrary morphism $f:q\to p$ in $\mathsf{Ab}/C$ with $f:A\to B$ in $\mathsf{Ab}$ the normal closure of $f$ over $C$ is carried by $\mathrm{Im}(f)$. The normal decomposition of $f$ in $\mathsf{Ab}$ is therefore carried by its decomposition in $\mathsf{Ab}$.

(3) The argument given in (2) obviously works in any abelian category (and beyond), in particular in the category $K\text{-}\mathsf{Vec}$ of $K$-vector spaces (with a field $K$). Let us give an alternative argument in this category, as follows. The embedding $A\hookrightarrow B$ of a subspace $A$ into a space $B$ splits since $B= A\oplus A'$ for some subspace $A'$, and its pushout along any linear map $q:A\to C$ is then described by the diagram
$$\xymatrix{ A\ar[d]\ar[rr]^q && C\ar[d]\\
A\oplus A'\ar[rr]^{\widetilde{q}=q\oplus 1_{A^{\prime}}} &&   C\oplus A'\\
}$$
whose vertical arrows are coproduct injections. Clearly, this pushout diagram is also a pullback. Therefore, like in $K\text{-}\mathsf{Vec}$ itself, the normal closure is discrete also in every slice  $K\text{-}\mathsf{Vec}/C$.
\end{examples}

We note that the argument given in $K\text{-}\mathsf{Vec}$ applies to many other categories, provided that we restrict the split monomorphism in question to being a coproduct injection. For a more general context, let us call
a category $\mathbb D$ with binary coproducts {\em pre-extensive} if every diagram
$$\xymatrix{A\ar[d]\ar[rr]^q && C\ar[d]\\
A+A' \ar[rr]^{q+1_{A'}}&& C+A'\\
}$$ whose vertical arrows  are coproduct injections is a pullback diagram in $\mathbb D$.
Every extensive category \cite{CLW93} and every quasi-abelian category (see Corollary \ref{quasi-abelian}) is pre-extensive. In particular, the categories $\SET, \TOP$ and $\AB$ are pre-extensive. Amongst these three categories, only $\SET$ has the property that every (regular) monomorphism is a coproduct injection. An appropriate generalization of this crucial property is given by the notion of {\em Boolean} category \cite{CLW93}, as an extensive category with a terminal object $1$ such that the first coproduct injection $t:1\to 1+1$ serves as a (regular) subobject classifier. Therefore, every (regular) subobject is a pullback of $t$ along its characteristic morphism and, as a pullback of a coproduct injection, it becomes a coproduct injection itself, by extensionality. Consequently:

\begin{corollary}\label{Boolean}
If the category $\ncatC$ is Boolean, then the normal closure in $\ncatC/C$ is discrete.
\end{corollary}
Of course, the Corollary applies in particular to $\CC=\SET$. Hence, like in $\mathsf{Ab}/C$, normal decompositions in $\SET/C$ are carried by the decomposition in the  underlying category, and they are perfect.

Let us now dualize the general statements for the slice $\CC/C$ and consider the coslice category
$$C/\CC=(\CC^{\op}/C)^{\op}$$
whose objects are $\CC$-morphisms with domain $C$, and whose morphisms $f:j\to k$ are given by $\CC$-morphisms $f:A\to B$ with $f\cdot j=k$. Denoting the normal closure of $f$ in $C/\CC$ by $\nu_{C/f}$ and its normal dual closure by $\pi_{C/f}$, we know by duality that only the latter closure is potentially interesting since the former is computed as in $\CC$:
$$\nu_{C/f}\cong \nu_f.$$
For the computation of $\pi_{C/f}$ we consider the diagram
$$\xymatrix{C\ar[rr]^j \ar[d]^s\ar@/_2.5pc/[dd]_{1_C}&& A\ar[d]^{\Delta}\ar@/_2.5pc/[dd]_(0.25){1_A}|(.5)\hole\ar@/^0.9pc/[rd]^{1_A} & \\
C\times_BA\ar[rr]^{j\times_B1_A\quad}\ar[d]^{p_f} && A\times_BA\ar[d]^{f_1}\ar[r]^{\quad f_2} &A\ar[d]^f \\
C\ar[rr]_j && A\ar[r]_f & B\\
}$$
and dualizes Proposition \ref{discreteincomma} to obtain:

\begin{corollary}
Let $f:j\to k=f\cdot j$ be a morphism in $C/\ncatC$ where $f:A\to B$ in $\CC$.

{\em (1)} The normal dual closure $\pi_{C/f}$ of $f$ in $C/\ncatC$ is obtained by the pushout of $p_f$ along $f_2\cdot\widetilde{j}$ in $\ncatC$, with $\widetilde{j}:=j\times_A1_A$; we write $A\to P_{C/f}$ for its underlying morphism in $\CC$.  If $f$ is a regular epimorphism and the left upper or lower rectangle of the diagram above is a pushout, then $\pi_{C/f}\cong f$.

{\em (2)} There is a unique regular epimorphism $\sigma:P_f\to P_{C/f}$ in $\CC$ with $\sigma\cdot\pi_f=\pi_{C/f}$ which compares the normal dual closure of $f$ taken in $C/\CC$ with its normal dual closure taken in $\CC$. Hence,
 $\pi_{f/C}\cong\pi_f$ in $A/ \CC$ if, and only if,
$\sigma$ is a monomorphism, and this holds whenever
(not only the right but also) the left square of the following diagram is a pushout:
$$\xymatrix{0\times_BA\ar[rr]^{!_C\times_B1_A}\ar[dd] && C\times_BA\ar[rr]^{f\cdot\widetilde{j}}\ar[dd]_(0.3){f'} && A\ar[dd]^{\pi_{f/C}}\ar[ld]_{\pi_f}\\
&&& P_f\ar[rd]^{\sigma} & \\
0\ar[rrru]|(.64)\hole\ar[rr] && C\ar[rr] && P_{f/C} \\
}$$
\end{corollary}

The dualization of Corollary \ref{normdecincomma} compares the normal dual decomposition in $C/\CC$ with that in $\CC$:
\begin{corollary}\label{decompincoslice}
For a morphism $f$ as in the Corollary above one has the commutative diagram 
$$\xymatrix{&& C\ar@/_2.3pc/[lldd]_j\ar@/^2.3pc/[rrdd]^k && \\
& P_{f}\ar[rr]^{\kappa_f}\ar[dd]_(0.65){\sigma}|(.5)\hole && N_f \ar[rd]^{\nu_f}& \\
A\ar[ru]^{\pi_{f}}\ar[rd]_{\pi_{C/f}}\ar[rrrr]^f&&&& B\\
& P_{C/f}\ar[rr]_{\kappa_{C/f}} && N_{C/f}\ar[ru]_{\nu_{C/f}}\ar[uu]_(0.65){\cong}|(.5)\hole &\\
}$$
\end{corollary}

\begin{examples}\label{cosliceSetRing}
(1) For any set $C$, not only the normal dual closure of a morphism $f$ in the slice $\SET/C$ may be formed as in $\SET$, but also its normal closure, namely by the image of $f$ (Corollary \ref{Boolean}).

However, for a morphism $f:j\to k$ in the coslice $C/\SET$ with $f:A\to B$ in $\SET$, only the normal closure is formed as in $\SET$. For its normal dual closure $\pi_{C/f}: A\to P_{C/f}$, one identifies precisely the pairs of points in those fibres of $f$ in $A$ which intersect the image of $j$, {\em i.e.}, setting
$$A_j:=f^{-1}(\mathrm{Im}(k))=f^{-1}(f(j(C)))$$
one may take
$$P_{C/f}:=\{f^{-1}(fx)\mid x\in A_j\}+ (A\setminus A_j)=\{f^{-1}(fx)\mid x\in \mathrm{Im}(j)\}+ (A\setminus A_j), $$
with $\pi_{C/f}: A\to P_{C/f}$ mapping $x\in A$ either to $f^{-1}(fx)$ or to $x$ itself, depending on whether $x\in A_j$. We note that, since $P_f\cong A$, here the map $\sigma:P_f\to P_{C/f}$ coincides with $\pi_{C/f}$.

A surjective morphism $f:j\to k$ in $C/\SET$ is a normal epimorphism precisely when $f$ maps $A\setminus A_j$ injectively. Such morphisms are closed under composition, making normal decompositions in $C/\SET$ perfect, with comparison morphisms characterized as the surjective maps $f$ mappping $A_j$ injectively.

Note that, for $C=1$, the coslice $1/\SET$ is nothing but the category $\SET_*$ of pointed sets, and the above reproduces the description of the normal dual closure as given in Examples \ref{firstexas}(4), as well as the statements of \ref{firstexamplesagain}(4).  Observe that even in this case $\sigma=\pi_{1/f}: A\to P_{1/f}$ fails to be bijective, unless the basepoint of $A$ is the only element in $A$ that $f$ maps to the basepoint of $B$.

(2) In the coslice category $C/\mathsf{Ab}$ with any object $C$, the normal dual closure is discrete, {\em i.e.}, the normal dual closure of a regular epimorphism $f$ is given by $f$ itself: $\pi_{C/f}\cong f$. This is true even when $\mathsf{Ab}$ is traded for $\mathsf{Grp}$, and the straightforward proof given for $\mathsf{C/Grp}$ in Proposition \ref{cosliceGrpdualclosure} works in particular in $C/\mathsf{Ab}$.

(3) The category $R\text{-}\mathsf{CAlg_1}$  of commutative unital $R$-algebras (for a commutative unital ring $R$) is equivalently described as the coslice category $R/\mathsf{CRing_1}$. For a morphism $f:A\to B$ in $R\text{-}\mathsf{CAlg_1}$, since its normal  closure in $\mathsf{CRing}_1$ may be given by the identity map on $B$ (Examples \ref{firstexas}(3)), this therefore holds in $R\text{-}\mathsf{CAlg_1}$ as well. Furthermore, since its normal {\em dual} closure $\pi_f$ in $\mathsf{CRing_1}$ can be taken to be the projection $A\to A/\mathrm{Ker}f$, the commutative diagram of Corollary \ref{decompincoslice} simplifies to
$$ \xymatrix{ & A/\mathrm{Ker}f\ar[rd]^{\kappa_f}\ar[dd]_(0.7){\sigma}|(.5)\hole & \\
A\ar[rr]^(0.7)f\ar[ru]^{\pi_f}\ar[rd]_{\pi_{R/f}} && B\\
& P_{R/f}\ar[ru]_{\kappa_{R/f}} & \\
}$$
Here, as a regular epimorphism in the coslice category $R/\mathsf{CRing_1}$, the map $\pi_{R/f}$ is also a regular epimorphism in $\mathsf{CRing}_1$, which makes $\sigma$ surjective. But $\sigma$ is also injective, as a first factor of $\kappa_f$. Consequently, the normal decomposition of $f$ in $R\text{-}\mathsf{CAlg_1}$ may be formed as in $\mathsf{CRing}_1$.
\end{examples}

\section{Slices and coslices of the category of T$_1$-spaces}
For a T$_1$-space $C$ we want to find the normal closure of a morphism $f:q\to p$ in $\TOP_1/C$, for a continuous map $p:B\to C$ of T$_1$-spaces. We first assume that $f:A\hookrightarrow B$ is the inclusion map of a subspace $A$ of $B$, so that $q=p|_A$ is a restriction of $p$, and our first task is to construct the pushout $B+_AC$ in $\TOP_1$.

We denote the topological closure of a subset $X\subseteq B$ in $B$ by $\overline{X}$. For every $a\in A$ we let 
$$E_a:=p^{-1}(pa)\cap A$$
be the $(p|_A)$-fibre of $pa$ in $A$ and observe:
\begin{itemize}
\item Since $C$ is a T$_1$-space, the $p$-fibre $p^{-1}(pa)$ of $pa\in C$ is closed in $B$ and, hence, $$\overline{E_a}\subseteq\overline{p^{-1}(pa)}=p^{-1}(pa).$$ 
\item Therefore, the distinct members of $\{\overline{E_a}\mid a\in A\}$ form a partition of the set
$$D:=\bigcup_{a\in A}\overline{E_a}\subseteq B,$$
where $(\overline{E_a}=\overline{E_{a'}}\iff pa=pa')$, for all $a,a'\in A$.
\end{itemize}
We now consider the (conveniently assumed-to-be) disjoint union
$$P:=(B\setminus D)+C$$
and define the mapping $i:B\to P$ by
$$ ib= \left\{ \begin{array}{rcl}b & \mbox{if} & b\in B\setminus D, \\
pa &\mbox{if} & b\in\overline{E_a} \mbox{ with } a \in A.\\
\end{array}\right.
$$
With $j$ denoting the inclusion map $C\hookrightarrow P$, we obtain the commutative diagram $$\xymatrix{A\ar[r]^{p|_A}\ar[d]_f & C\ar[d]^j\\ B\ar[r]^i & P}$$
and let $P$ carry the finest topology making the maps $i$ and $j$ continuous. Then:

\begin{proposition}\label{T1pushout}
The above diagram is a pushout in $\TOP_1$.
\end{proposition}
\begin{proof}
To show that $P$ is a T$_1$-space, it suffices to see that the $i$- and $j$-fibres of any $z\in P$ are closed in $B$ and $C$, respectively. But this is obvious since, for all $z\in P$,
$$i^{-1}z=\left\{ \begin{array}{rcl} \{z\} & \mbox{ if } & z\in B\setminus D,\\
\overline{E_a} & \mbox{ if } & z=pa \mbox{ with } a\in A,\\
\emptyset & \mbox{ if } & z\in C\setminus p(A);\\ 
\end{array}\right.
\qquad j^{-1}z=\left\{ \begin{array}{rcl} \emptyset & \mbox{ if } & z\in B\setminus D,\\
\{z\} & \mbox{ if } & z\in C.\\
\end{array}\right.
$$
Regarding the universal property, consider any continuous maps $g: B\to Y, h:C\to Y$ into a T$_1$-space $Y$ such that $ga=h(pa)$ for all $a\in A$. Since for any $a'\in E_a$ one has $pa=pa'$ and then $ga'=h(pa')=h(pa)=ga$, so that $g(E_a)=\{ga\}$, the T$_1$-property of $Y$ forces $g$ to be constant even on $\overline{E_a}$; indeed, denoting the topological closure also in $Y$ in a standard manner, one has
$$g(\overline{E_a})\subseteq \overline{g(E_a)}=\overline{\{ga\}}=\{ga\}.$$
Therefore, when we can define the map $t:P\to Y$ by $t|_{B\setminus D}= g|_{B\setminus D}$ and $t|_C=h$, it actually satisfies $t\cdot i=g$ and, trivially, $t\cdot j=h$. Indeed, for every $b\in \overline{E_a}\subseteq D$ with $a\in A$ we compute
$$t(ib)=t(pa)=h(pa)=ga=gb,$$
so that $t|_D=g|_D$.
Consequently, by the definition of the topology on $P$, the map $t$ is continuous and trivially uniquely determined by the constraints $t\cdot i=g$ and $t\cdot j=h$.
\end{proof}

\begin{corollary}\label{T1corollary}
In the setting of the Proposition one has $N_{f/C}=\bigcup_{a\in A}\overline{E_a}$.
\end{corollary}
\begin{proof}
By definition, the normal closure of $f$ is obtained by pulling back the pushout injection $j$ along $i$ which, in the notation of Section 6, gives
$$N_{f/C}= i^{-1}(C)=\{b\in B\mid ib\in C\}=D=\bigcup_{a\in A}\overline{E_a}.$$
\end{proof}
Here is an easy example showing that one may have strict inclusions $A\subset N_{f/C}\subset N_f=\overline{A}$:
\begin{example}
In the Euclidean plane $\mathbb R^2$ provided with the (first) projection $p:\mathbb R^2\to \mathbb R$, consider the open disc $A$ of radius $1$ about the origin, and let $f$ be its inclusion map into $\mathbb R^2$. Then $N_{f/\mathbb R}=\overline{A}\setminus\{-1,1\}$, where $\overline{A}$ is the closed disc of radius $1$ about the origin.
\end{example}

\begin{theorem}\label{slicedT1closure}
For a pair of continuous maps $f:A\to B$ and $p:B\to C$ of T$_1$-spaces, regarded as a morphism $f: q\to p$ in $\TOP_1/C$ with $q=p\cdot f$, the normal closure of $f$ over $C$ may be computed as
$$\nu_{f/C}: N_{f/C}=\bigcup_{a\in A}\overline{f(q^{-1}(qa))}\hookrightarrow B.$$
The $\TOP_1/C$-morphism $f:q\to p$ is a normal monomorphism if, and only if, $f$ is an embedding such that, for every $a\in A$, the set $f(q^{-1}(qa))$ is closed in $B$. The class $\mathsf{NMono}(\TOP_1/C)$ is closed under composition.
\end{theorem}
\begin{proof}
By Lemma \ref{triviallemma}, the morphism $f$ has the same normal closure as the subspace inclusion $f(A)\hookrightarrow B$ in $\TOP_1/C$, so that the Corollary gives
$$N_{f/C}=\bigcup_{a\in A}\overline{p^{-1}(p(fa))\cap f(A)}.$$ Since $p^{-1}(qa)\cap f(A)=f(q^{-1}(qa))$ for all $a\in A$, the claimed formula follows.

For the characterization of normal monomorphisms in $\TOP_1/C$, first assume that $f:q\to p$ is such. Since the normal dual closure of $f$ in $\TOP_1/C$ is, like in $\TOP_1$, trivial (Examples \ref{firstexamplesagain}(6)), we may take $\pi_{f/C}$ to be the identity map on $A$ and obtain that the comparison map $\kappa_{f/C}:A\to N_{f/C},\,a\mapsto fa,$ must be a homeomorphism. This makes $f=\nu_{f/C}\cdot\kappa_{f/C}$ an embedding. Furthermore, any $y\in\overline{f(q^{-1}(qa))}^B$ with $a\in A$ must be of the form $y=fx$ with $x\in A$, so that
$$qx=py\;\in\;\;\overline{p(f(q^{-1}(qa)))}^C=\overline{q(q^{-1}(qa)}^C=\overline{\{qa\}}^C=\{qa\}\,.      $$
Consequently, $x\in q^{-1}(qa)$, and $y\in f(q^{-1}(qa))$ follows. This shows that $f(q^{-1}(qa))$ is closed in $B$.

Conversely, if $f$ is an embedding with every set $f(q^{-1}(qa))$ closed in $B$,
then $\kappa_{f/C}: A\to N_{f/C}$ must be an embedding as well, and $N_{f/C}=\bigcup_{a\in A}f(q^{-1}(qa))$ holds. This equality makes $\kappa_{f/C}$ surjective and, hence a homeomorphism.

If $f:p\to q$ and $g:q\to r$ are both normal monomorphisms in $\TOP_1/C$, with $r:D\to C$ in $\TOP_1$, we have subspace embeddings $f:A\to B$ and $g:B\to D$ and may assume that these are inclusion maps. Then all, $A, B$ and every $q^{-1}(qa)\;(a\in A)$, are subspaces of $D$, with $q^{-1}(qa)=r^{-1}(ra)\cap A$ closed in $B$ and $p^{-1}(pa)=r^{-1}(ra)\cap B$ closed in $C$, which makes $q^{-1}(qa)$ also closed in $D$.
\end{proof}

\begin{corollary}
The normal decomposition of $f:q\to p$ in $\TOP_1/C$ factors its underlying continuous map $f:A\to B$ as
$$\xymatrix{A\ar[rr]^{\pi_{f/C}\cong 1_A} && A \ar[rr]^{\kappa_{f/C}\qquad\quad} && \bigcup_{a\in A}\overline{f(q^{-1}(qa))}\ar[rr]^{\qquad\quad\nu_{f/C}} &&B}.$$
These decompositions are perfect, and $f$ is a comparison morphism if, and only if, $B=p^{-1}(q(A))$ and $f$ is a fibrewise dense map in $\TOP_1$, that is: for every $a\in A$, the image  of the restriction $q^{-1}(qa)\longrightarrow p^{-1}(qa)$ of $f$ is dense in its codomain.
\end{corollary}
\begin{proof}
The class $\mathsf{NEpi}(\TOP_1/C)$ is the class of isomorphisms and, hence, trivially closed under composition, and so is  $\mathsf{NMono}(\TOP_1/C)$, by the Theorem. This makes the normal decompositions in $\TOP_1/C$ perfect.

Since the normal dual closure $\pi_f$ of every morphism $f$ in $\TOP_1/C$ is an isomorphism, $f$ is a comparison morphism precisely when $\nu_f$ is an isomorphism, {\em i.e.}, when $B=N_{f/C}$. Under this condition every $y\in B$ lies in $\overline{f(q^{-1}(qa))}^B$ for some $a\in A$, and (as in the Proof of the Theorem) $py\in \overline{p(f(q^{-1}(qa)))}^C=\{qa\}$ and then $y\in p^{-1}(qa)\subseteq p^{-1}(q(A))$ follows. This shows $B=p^{-1}(q(A))$. Furthermore, we note that, for {\em any} $a\in A$, since the fibre $p^{-1}(qa)$ is always closed in $B$, the set
$\overline{f(q^{-1}(qa))}^B$ is actually the closure of $f(q^{-1}(qa))$ in $p^{-1}(qa)$. As the distinct members of $\{p^{-1}(qa)\mid a\in A\}$ form a partition of $B=p^{-1}(q(A))=N_{f/C}$, we conclude that $\overline{f(q^{-1}(qa))}^B=p^{-1}(q(a))$ holds for every $a\in A$.

This makes the stated conditions for $f$ being a comparison map necessary. Their sufficiency is trivial; indeed, if we have $\overline{f(q^{-1}(qa))}^B=p^{-1}(qa)$ for all $a\in A$ with $p^{-1}(q(A))=B$, taking their union gives $N_{f/C}=B$.
\end{proof}

\begin{corollary}
For any T$_1$-space $C$,  $(\{\text{comparison morphism}\}, \{\text{normal mono }\})$ is an orthogonal factorization system of $\TOP_1/C$. For $C=1$, it gives the $(\{\text{dense}\},\{\text{closed embedding}\})$-system of $\TOP_1$.
\end{corollary}

We now turn to the coslices of $\TOP_1$ for any T$_1$-space $C$. Unlike in the case of the normal closure in the slices of $\TOP_1$, for the coslices one easily proves:
\begin{proposition}
The normal dual closure of a morphism in any coslice $C/\TOP_1$ may be computed by providing its construction in $C/\SET$ with the pushout topology.
\end{proposition}
\begin{proof}
For a morphism $f:j\to k$ in $C/\TOP_1$ with a continuous map $f:A\to B$ of T$_1$-spaces, with $A_j=f^{-1}(\mathrm{Im}(k))$ one considers the set
$$P_{C/f}=\{f^{-1}(fx)\mid x\in A_j\}+ (A\setminus A_j), $$
of Examples \ref{cosliceSetRing}(1). At the $\SET$-level,
 it comes with the pushout maps
$$\pi_{C/f}: A\longrightarrow P_{C/f}\qquad\text{and}\qquad\pi_{C/f}\cdot j:C\longrightarrow  P_{C/f}.$$
The quotient topology on $P_{C/f}$ induced by $\pi_{C/f}$ makes also the map $\pi_{C/f}\cdot j$ continuous, so that $P_{C/f}$ becomes a pushout in $\TOP$. But $P_{C/f}$ actually lies in $\TOP_1$, as an easy inspection of the fibres of $\pi_{C/f}$ (like in the Proof of Proposition \ref{T1pushout}) reveals.
\end{proof}

In conjunction with Examples \ref{firstexamplesagain}(6) and maintaining the above notation, we may now conclude in  straightforward manner:
\begin{corollary}
The normal decomposition of $f:j\to k$ in $C/\TOP_1$ factors its underlying continuous map $f:A\to B$ as
$$\xymatrix{A\ar[rr]^{\pi_{C/f}} && P_{C/f}\ar[rr]^{\kappa_{C/f}} && \overline{f(A)}\ar[rr]^{\nu_{C/f}\,\cong\,\nu_f} &&B}.$$
The morphism $f$ is a normal epimorphism in $C/\TOP_1$ if, and only if, it is a quotient map which maps $A\setminus f^{-1}(\mathrm{Im}(k))$ injectively, and $f$ is a normal monomorphism if, and only if, it is a closed embedding.  As both of these types of morphisms are closed under composition, the normal decompositions in $C/\TOP_1$ are perfect, and
$f$ is a comparison map if, and only if, $f$ has a dense image in $B$ and maps $f^{-1}(\mathrm{Im}(k))$ injectively. 
\end{corollary}

\section{Slices and coslices of the category of groups}
Proceeding similarly as in the previous section, for a group $C$ we want to find the normal closure $N_{f/C}\hookrightarrow B$ of a morphism $f:q\to p$ in $\mathsf{Grp}/C$, where $p:B\to C$ in $\mathsf{Grp}$ is arbitrary and $f$ is first assumed to be a subgroup inclusion map $A\hookrightarrow B$, so that $q=p|_A$ is a restriction of $p$. Our goal is to find an acceptable description of the normal closure of $f$ as given by the subgroup $N_{f/C}$ of $B$, avoiding recourse to  the pushout $B+_AC$ in $\mathsf{Grp}$ to the largest extent possible.

To this end we form the subgroup
$$E:=\mathrm{Ker}(p)\cap A= \mathrm{Ker}(p|_A)$$
of $B$ and denote by
 $\widehat{E}^B$
the least normal subgroup of $B$ containing $E$. Now consider the diagram
$$\xymatrix{A\ar[r]^{\widetilde{p}\qquad}\ar[d]\ar@/^2.0pc/[rr]^{p|_A} & A/E\cong p(A)\;\ar[r]\ar[d]^k & C\\
B\ar[r]^{\overline{p}}  & B/\widehat{E}^B & \\
}$$
where $\widetilde{p}$ and $\overline{p}$ are projection maps and $k$ is defined by the assignment $aE\longmapsto a\widehat{E}^B$ or, alternatively, by $p(a)\longmapsto \overline{p}(a)$, depending on whether we see the domain of $k$ to be $A/E$ or $p(A)$. Either way, the map $k$ is well defined since $E\subseteq \widehat{E}^B\cap A=\mathrm{Ker}(\overline{p}|_A)$, and it makes the left rectangle of the diagram commutative. In fact:

\begin{proposition}\label{Grpsliceprop}
The rectangle of the above diagram is a pushout in $\mathsf{Grp}$. Moreover, the homomorphism $k$ is injective, so that $\mathrm{Im}(k)\cong p(A)$, and one has
$$\overline{p}^{-1}(\mathrm{Im}(k))=A\widehat{E}^B\leq B.$$
\end{proposition}

\begin{proof}
For any group homomorphisms $g:B\to X$ and $h: \widetilde{p}(A)\to X$ with $g|_A=h\cdot\widetilde{p}$, the group
$$g(E)=g(\mathrm{Ker}(\widetilde{p}))=h(\widetilde{p}(\mathrm{Ker}(\widetilde{p}))$$
is trivial, {\em i.e.}, $E\subseteq \mathrm{Ker}(g)$. Therefore, $g$ factors uniquely through $\overline{p}$, by a homomorphism $B/\widehat{E}^B\to X$ which, by surjectivity of $\widetilde{p}$, also makes $h$ factor through $k$. This shows that the rectangle is a pushout diagram.

To see that $\mathrm{Ker}(k)$ as a subgroup of $p(A)$ contains only the neutral element $e_C$ of $C$, we note that
$$\mathrm{Ker}(k)=\{pa\mid a\in A,\overline{p}(a)=\widehat{E}^B\}=p(A\cap\widehat{E}^B)\subseteq p(\widehat{E}^B)\subseteq\widehat{p(E)}^C=\{e_C\},$$
where we use the fact that taking normal hulls of subgroups is a (categorical) closure operator \cite{DT95}, and that $E\subseteq\mathrm{Ker}(p)$. For the computation of the inverse image of $\mathrm{Im}(k)$ under $\overline{p}$, we have
\begin{align*}
\overline{p}^{-1}(\mathrm{Im}(k)) & =\{y\in B\mid\exists a\in A: y\widehat{E}^B=a\widehat{E}^B\} \\
& =\{y\in B\mid\exists a\in A: y\in a\widehat{E}^B\} = A\widehat{E}^B.
\end{align*}
\end{proof}

Moving from a subspace inclusion to any homomorphism $f$ over $C$, we can now state:
\begin{theorem}\label{slicedGrpclosure}
The normal closure $\nu_{f/C}$ of a morphism $f:q\to p$ in $\mathsf{Grp}/C$ with homomorphisms $f:A\to B$ and $p: B\to C$ is given by the subgroup inclusion
$$N_{f/C}=\mathrm{Im}(f)\widehat{E}^B\hookrightarrow B,$$
with the subgroup $\mathrm{Im}(f)$ of $B$ and $\widehat{E}^B$ denoting the least normal subgroup of $B$ containing $E=\mathrm{Ker}(p)\cap\mathrm{Im}(f)$. It coincides with the normal closure $\nu_f$ of $f$ in $\mathsf{Grp}$ if $\mathrm{Im}(f)\subseteq \mathrm{Ker}(p)$.
\end{theorem}

\begin{proof}
We form the pushout of $f$ along $p\cdot f$ in $\mathsf{Grp}$ given by the outer arrows of the diagram
$$\xymatrix{A\ar@/^1.5pc/[rrr]^{p\cdot f}\ar@{->>}[r]\ar[d]_f & \mathrm{Im}(f)\ar@{->>}[r]^{\widetilde{p}\quad}\ar@{>->}[d] & \mathrm{Im}(p\cdot f)\;\ar@{>->}[d]^k\ar@{>->}[r]& C\ar@{>->}[d]^{j_f}\\
B\ar@{=}[r]\ar@/_1.5pc/[rrr] & B\ar@{->>}[r]^{\overline{p}} & B/\widehat{E}^B\;\ar@{>->}[r]^i & B+_A C\\
}$$

By the surjectivity of $\xymatrix{A\ar@{->>}[r] & \mathrm{Im}(f)}$ one has $B+_AC\cong B+_{\mathrm{Im}(f)}C$, {\em i.e.}, as done earlier without loss of generality, we may assume $A=\mathrm{Im}(f)$, and then our notation for $E, \widetilde{p}, \overline{p}$ and $k$ is as in Proposition \ref{Grpsliceprop}. From there we know that the middle rectangle of the diagram above is a pushout diagram, thus giving us the arrow $i$ completing the entire commutative diagram. Moreover, by the surjectivity of $A\to \mathrm{Im}(p\cdot f)$, also the right inscribed rectangle is a pushout diagram. 

Now, since by Schreier's classical result \cite{Schreier27, Neumann54} $\mathsf{Grp}$ has the so-called {\em strong amalgamation property} \cite{KMPT83}, in this pushout of a pair of monomorphisms also the pushout maps $i$ and $j_f$ are monomorphisms (in fact, $j_f$ splits, with retraction $r=\langle p,1_C\rangle$ -- see Section 6) and, moreover, the pushout diagram is also a pullback diagram; that is, treating $k$ as a subgroup embedding, one may construct $B+_AC$ as the free product of  $B/\widehat{E}^B$ and $C$ with an amalgamated subroup $\mathrm{Im}(p\cdot f)$. Therefore, the pullback of $j_f$ along $i\cdot\overline{p}$ can be computed as the pullback of $k$ along $\overline{p}$ and, by Proposition \ref{Grpsliceprop}, that pullback gives the subgroup inclusion
$N_{f/C}:=\mathrm{Im}(f)\widehat{E}^B\hookrightarrow B$, as claimed.

Trivially, when $\mathrm{Im}(f)\subseteq \mathrm{Ker}(p)$ one has $E=\mathrm{Im}(f)$, and  $N_{f/C}=\widehat{E}^B=\widehat{\mathrm{Im}(f)}^B=N_f$ follows.
\end{proof}

\begin{remark}
It is interesting to observe that, since one has the group homomorphisms $j_f$ and $r$ with $r\cdot j_f=1_C$, the pushout $B+_AC$ can be presented as a {\em semidirect product} $X\rtimes C$, with $C\cong\mathrm{Im}(j_f)\leq B+_AC$ acting on the group $X:=\mathrm{Ker}(r)\trianglelefteq B+_AC$ by conjugation:
$$B+_AC\cong X\rtimes C\,;$$
 see, for example, Section 5.2 of \cite{BorceuxBourn04}.
\end{remark}

Maintaining the notation of the Theorem and considering also Examples \ref{firstexamplesagain}(1) we conclude:

\begin{corollary}
The normal decomposition of $f:q\to p$ in $\mathsf{Grp}/C$ factors its underlying homomorphism $f:A\to B$ as
$$\xymatrix{A\ar[rr]^{\pi_{f/C}\cong\pi_f\qquad} && A/\mathrm{Ker}(f)\ar[rr]^{\kappa_{f/C}} && \mathrm{Im}(f)\widehat{E}^B\ar[rr]^{\;\nu_{f/C}} &&B}.$$
The morphism $f$ is a normal epimorphism in $\mathsf{Grp}/C$ if, and only if, it is a quotient map in $\mathsf{Grp}$, and it is a normal monomorphism if, and only if, the normal hull of $E=\mathrm{Ker}(p)\cap\mathrm{Im}(f)$ in $B$ is contained in $\mathrm{Im}(f)$. For any $C$, the normal decomposition fails to be perfect.
\end{corollary}

For the failure of perfectness note that $\mathsf{Grp}$ embeds fully into $\mathsf{Grp}/C$ since every group $A$ is the domain of the trivial morphism $A\to C$.

We finally turn to the coslice category $C/\mathsf{Grp}$ and prove that, for any group $C$, the normal dual closure is discrete:

\begin{proposition}\label{cosliceGrpdualclosure}
Let $f:j\to k$ be a morphism in $C/\mathsf{Grp}$, with group homomorphisms $j:C\to A,\;f:A\to B$, and $k=f\cdot j$. Then the normal dual closure $\pi_{C/f}$ of $f$ is carried by the projection
$A\to A/\mathrm{Ker}(f)\cong\mathrm{Im}(f)$.
\end{proposition}

\begin{proof}
The coslice $C/\mathsf{Grp}$ inherits its (regular epi, mono)-factorizations from $\mathsf{Grp}$. Hence, by Lemma \ref{triviallemma}, without loss of generality we may assume that $f:A\to B$ be surjective. It then suffices to prove that the pullback diagram
$$\xymatrix{C\times_BA\ar[r]^{\quad r}\ar[d]_{p} &A\ar[d]^f\\
C\ar[r]^{f\cdot j} &  B\\
}$$
is also a pushout in $\mathsf{Grp}$. Indeed, for any homomorphisms $g:C\to X,\,h:A\to X$ with $g\cdot p=h\cdot r$ one has $g(z)=h(x)$ whenever $z\in C$ and $x\in A$ satisfy $f(jz)=f(x)$. For $z$ neutral in $C$, this shows $\mathrm{Ker}(f)\subseteq\mathrm{Ker}(h)$, so that $h$ factors through $f$, by a (unique) homomorphism $t:B\to X$. Since with $f$ also $p$ is surjective, $t$ satisfies also $t\cdot f\cdot j=g$, as required.
\end{proof}

\begin{corollary}
For any group $C$, normal decompositions in $C/\mathsf{Grp}$ are formed as in $\mathsf{Grp}$. Therefore, a morphism in $C/\mathsf{Grp}$ is a normal mono- or epimorphism if, and only if, it has the respective property in $\mathsf{Grp}$, and normal decompositions in $C/\mathsf{Grp}$ fail to be perfect.
\end{corollary}

\section{Concluding remarks on the pullback-pushout adjunction}

We finally point out that the pullback-pushout adjunction between cospans and spans in a category provides a common generalized setting for the formation of normal closures and their duals. Ignoring the well-studied broader bi- or double-category-theoretic context, we recall here only that, for a category $\ncatC$ and fixed objects $X$ and $Y$  in $\ncatC$, one has the (ordinary) categories
$$\Span_{\ncatC}(X,Y)\qquad\text{and}\qquad\Cosp_{\ncatC}(X,Y)$$
of spans $X\leftarrow\bullet\rightarrow Y$ and cospans $X\rightarrow\bullet\leftarrow Y$ of arrows in $\ncatC$; their morphisms $s:(u,v)\to(u',v')$ and $t:(p,q)\to(p',q')$ are respectively given by the commutative diagrams
$$\xymatrix{& A\ar[ld]_u\ar[rd]^v\ar[dd]^s & && & B\ar[dd]^t &\\
X && Y &\text{and}& X\ar[ru]^p\ar[rd]_{p'} && Y\ar[lu]_q\ar[ld]^{q'}\\
& A'\ar[lu]^{u'}\ar[ru]_{v'} & && & B'  &\\
}$$
in $\ncatC$. When $\ncatC$ has (tacitly chosen) pullbacks and pushouts, we have the functors
$$ \Po:\Span_{\ncatC}(X,Y)\to\Cosp_{\ncatC}(X,Y)\qquad\text{and}\qquad \Pb:\Cosp_{\ncatC}(X,Y)\to\Span_{\ncatC}(X,Y),$$
respectively assigning to a span its pushout cospan and to a cospan its pullback span. Clearly:

\begin{proposition}\label{pbpoadj}
In the presence of pullbacks and pushouts in $\ncatC$ one has the adjunction $\Po\dashv \Pb$ inducing an idempotent monad on $\Span_{\ncatC}(X,Y)$. Its category of algebras is therefore equivalently described as the full subcategory of {\em Doolittle spans}, i.e., of those spans $(u,v)$ for which the unit of the adjunction gives an isomorphism $\eta:(u,v)\to\Pb(\Po(u,v))$. Dually, the {\em Doolittle cospans}, i.e., the coalgebras of the idempotent comonad induced by the adjunction $\mathrm{Po}\dashv\mathrm{Pb}$, are given by those cospans $(p,q)$ for which the counit $\varepsilon: \mathrm{Po}(\mathrm{Pb}(p,q))\to(p,q)$ is an isomorphism.
\end{proposition}

The proof is a routine exercise. Idempotency follows from $\Po\cdot \Pb\cdot\Po\cong\Po$. Squares that are simultaneously pullback and pushout diagrams were called \textit{Doolittle diagrams} by Peter Freyd; in \cite{AHS90} they are called \textit{pulation squares}. We mention here  only one elementary example.

\begin{example}\label{Dolittleexample}
For $\ncatC=\AB$ the category of abelian groups, we present the pushout of $(u:A\to X, v:A\to Y)$ by $B:= (X\oplus Y)/\mathrm{Im}(u-v)$, thus considering $u$ and $v$ as $(X\oplus Y)$-valued. The  pullback of the restricted projections $(p:X\to B, q:Y\to B)$ may be presented by the subgroup $\mathrm{Im}(u-v)$ of $X\oplus Y$. Since the canonical morphism $\eta: A\to \mathrm{Im}(u-v),$ with $a\mapsto ua+va$ is surjective, $(u,v)$ is a Doolittle span precisely when $(u,v)$ is a monic pair, {\em i.e.}, when $\mathrm{Ker}u\,\cap\, \mathrm{Ker}v=0$.

Dually, given a cospan $(p:X\to B, q:Y\to B)$, we have the projections $u,v$ of its pullback $A=\{(x,y)\mid px=qy\}$ and form their pushout $X\oplus Y/\mathrm{Im}(u-v)$. The canonical map $\varepsilon:X\oplus Y/\mathrm{Im}(u-v)\to B$ sending the coset of $(x,y)$ to $px+qy$ has a trivial kernel. This makes $(p,q)$ a Doolittle cospan precisely when $\varepsilon$ is surjective, {\em i.e.}, when $B=\mathrm{Im}p+\mathrm{Im}q$.
\end{example}


For all objects $A,B$ in our finitely complete and finitely cocomplete category $\CC$, the Proposition gives in particular the adjunctions
$$\xymatrix{
\CC/B\ar[r]^{\cong \quad} & \mathsf{Span}(B,1)\ar@/^0.2pc/[r]^{\mathrm{Po}}& \mathsf{Cosp}(B,1)\ar@/^0.2pc/[l]^{\mathrm{Pb}} && \mathsf{Span}(0,A)\ar@/^0.2pc/[r]^{\mathrm{Po}} & \mathsf{Cosp}(0,A)\ar[r]^{\quad\cong}\ar@/^0.2pc/[l]^{\mathrm{Pb}} & A/\CC
}$$
inducing an idempotent monad on $\CC/B$ and an idempotent comonad on $A/\CC$, respectively. For a morphism $f:A\to B$ in $\CC$, considered as an object of $\CC/B$ or of $A/\CC$, the unit $\eta$ of the first adjunction at $f$ gives (in the notation of Section 3) the ($\CC/B$)-morphism $f\to\nu_f$ which is carried by $\hat{f}:A\to N_f$ in $\CC$, and the counit $\varepsilon$ of the second adjunction at $f$ gives the $A/\CC$-morphism $\pi_f\to f$, carried by $\check{f}:P_f\to B$ in $\CC$.

Hence, from this perspective, $\mathrm{Pb}(\mathrm{Po}(u,v))$ could be considered as the normal closure of an arbitrary span $(u,v)$, and $\mathrm{Po}(\mathrm{Pb}(p,q))$ as the normal dual closure of a cospan $(p,q)$, but this generalized categorical environment no longer allows for the formation of comparison morphisms in the same category.

$$ $$

{\small Dept. of Mathematics and Applied Math., University of the Free State, Bloemfontein, South Africa\\ jansenrs@ufs.ac.ca

Dept. of Mathematics, School of Natural Sciences, National University of Sciences and Technology, Islamabad, Pakistan, and 
Dept. of Mathematics and Statistics, York University, Toronto ON, Canada\\
qasim99956@gmail.com

Dept. of Mathematics and Statistics, York University, Toronto ON, Canada\\
tholen@yorku.ca}

\end{document}